\documentclass[11pt]{amsart}
\usepackage[margin=2.5cm]{geometry}
\usepackage{amssymb}
\usepackage{bm}
\usepackage[centertags]{amsmath}
\usepackage{amsfonts}
\usepackage{amsthm}
\usepackage{mathrsfs}
\usepackage{mathtools,xfrac}
\usepackage[usenames,dvipsnames,svgnames,table]{xcolor}
\usepackage[mathscr]{euscript}

\DeclareMathOperator\w{w}
\DeclareMathOperator\rng{rng}

\DeclareMathOperator\supp{supp}
\DeclareMathOperator\sgn{sgn}
\DeclareMathOperator\suc{succ}
\DeclareMathOperator{\dist}{dist}
\DeclareMathOperator{\spa}{span}

\newcommand{\mt}{\mathcal{T}}

\newcommand{\bfg}{\mathbf{g}}

\numberwithin{equation}{section}

\newtheorem{theorem}{Theorem}[section]
\newtheorem*{theorem*}{Main Theorem}
\newtheorem{lemma}[theorem]{Lemma}
\newtheorem{proposition}[theorem]{Proposition}
\newtheorem{corollary}[theorem]{Corollary}

\newtheorem{definition}[theorem]{Definition}

\newtheorem{notation}{Notation}

\newtheorem{remark}[theorem]{Remark}

\newcommand{\N}{\mathbb{N}}
\newcommand{\R}{\mathbb{R}}

\usepackage[utf8]{inputenc}

\author{Anna Pelczar-Barwacz}
\thanks{The research of the author was supported by the grant of the National Science Centre (NCN), Poland, no. UMO-2020/39/B/ST1/01042}

\email{anna.pelczar@uj.edu.pl}
\address{Jagiellonian University, Faculty of Mathematics and Computer Science, Institute of Mathematics,  {\L}ojasiewicza 6, 30-348 Krak\'ow,  Poland}

\usepackage[T1]{fontenc}
\numberwithin{subsection}{section}
\numberwithin{equation}{section}

\begin{document}
\title{A Banach space with an  infinite dimensional reflexive quotient algebra $\mathscr{L}(X)/\mathscr{SS}(X)$} 

\begin{abstract}
We construct a Banach space $X$ of Gowers-Maurey type such that the algebra of bounded operators $\mathscr{L}(X)$ is a direct sum of an  infinite dimensional reflexive Banach space and the operator ideal of strictly singular operators $\mathscr{SS}(X)$. 
\end{abstract}
\maketitle

\section*{Introduction}

The study of the structure of the closed operator ideals in the Banach algebra $\mathscr{L}(X)$ on a Banach space $X$, started in early 1940's,  developed  intensively in recent years. One of the main direction of research concentrates on determining Banach algebras which can be represented as a quotient algebra of $\mathscr{L}(X)$ for some Banach space $X$, with special emphasis on the Calkin algebra of $X$, that is $\mathscr{C}(X)=\mathscr{L}(X)/\mathscr{K}(X)$, where $\mathscr{K}(X)$ denotes the ideal of compact operators. Concerning this direction, powerful examples solving several long-open structural questions on Banach spaces were built mainly within the framework of Gowers-Maurey and Argyros-Haydon type spaces theory. 

The Gowers-Maurey method,  starting with techniques introduced by Th.~Schlumprecht \cite{s1} and developed widely, provides a list of Banach spaces sharing many extreme properties, among others with an explicitly described quotient algebra $\mathscr{L}(X)/\mathscr{SS}(X)$, where $\mathscr{SS}(X)$ denotes the closed operator ideal of  strictly singular operators on $X$ (cf.  \cite{gm,gm2,ad}). Recall that an operator $S$ between Banach spaces $U,V$ is called strictly singular if there is no infinite dimensional subspace $W\subset U$ such that the restriction $S|_W:W\rightarrow S(W)$ is an isomorphism. The Gowers-Maurey space $X_{GM}$ provided a negative 
answer to the unconditional
basic sequence problem and turned out to satisfy a much stronger
property than lack of unconditional basic sequences, i.e.  hereditary indecomposability (HI; no two closed infinite dimensional subspaces of $X_{GM}$ form a direct sum)  \cite{gm}. In close relation to this fact, the quotient algebra $\mathscr{L}(X_{GM})/\mathscr{SS}(X_{GM})$ is one-dimensional \cite{gm} 
(whereas $\mathscr{SS}(X_{GM})$ itself has rich structure, containing $2^\mathfrak{c}$ many pairwise distinct closed operator ideals, \cite{mp}). The same result holds also in the case of $\ell_1$-asymptotic HI space of Argyros-Deliyanni \cite{ad}; in general for any HI space $X$ the quotient algebra $\mathscr{L}(X)/\mathscr{SS}(X)$ is isomorphic  either to $\mathbb{R}$, $\mathbb{C}$ or $\mathbb{H}$  \cite{f}. 

The general theory developed in \cite{gm2}, cf. also \cite{ma}, proves in particular that the convolution algebra $\ell_1(\mathbb{Z})$, among other examples, can be represented as the quotient algebra $\mathscr{L}(X)/\mathscr{SS}(X)$ for some Banach spaces $X$, with the left and right shifts generating $\mathscr{L}(X)$ modulo strictly singular operators. Recently P.Motakis \cite{mo}, within the framework of Gowers-Maurey spaces, for any compact metric $K$ built a space with the quotient algebra $\mathscr{L}(X)/\mathscr{SS}(X)$ isometric to the space $C(K)$ of continuous functions on $K$, using diagonal operators on $X$. Examples of non-separable spaces $X=C(K)$ with the quotient algebra  $\mathscr{L}(X)/\mathscr{SS}(X)$ isometric to $X$  were constructed in \cite{k,p}. 

The results concerning Gowers-Maurey type spaces were lifted and augmented on the level of Argyros-Haydon type spaces, with the ideal $\mathscr{SS}(X)$ replaced by $\mathscr{K}(X)$. The construction of the celebrated Argyros-Haydon space $X_{AH}$, adapting the Gowers-Maurey method in the setting of  Bourgain-Delbaen $\mathscr{L}_\infty$-spaces, solves the fundamental "scalar+compact" problem \cite{ah}, that is yields a one-dimensional Calkin algebra $\mathscr{C}(X_{AH})$ \cite{ah}. The space $X_{AH}$ served as the basis, in the process of combining finitely many carefully chosen variants of $X_{AH}$, certain iterative constructions of Argyros-Haydon sums of variants of $X_{AH}$ or similar procedures,  for building other examples of Banach spaces with the Calkin algebra of prescribed form, as any finite dimensional semi-simple complex algebra \cite{kl}, any space $C(K)$ with countable compact $K$ \cite{mpz} or any non-reflexive Banach spaces with an unconditional basis, as well as some HI spaces \cite{mpt}.  

On the other hand, direct lifting constructions of Gowers-Maurey type spaces with specific quotient algebras (over the ideal of strictly singular operators) to the setting of Bourgain-Delbaen spaces provides also  powerful examples, as a Banach space with Calkin algebra isomorphic to the convolution algebra $\ell_1(\N_0)$ \cite{t2}, whose construction follows the approach of \cite{gm2}, or, for any compact $K$, a Banach space with Calkin algebra isometric to space $C(K)$ \cite{mo}. 

In the context of the results described above the natural question concerns the existence of a Banach space with an infinite dimensional reflexive Calkin algebra, cf. \cite{mo, mpt, mpz}. As it was mentioned above, constructions of Gowers-Maurey type spaces are often adaptable to the setting of Argyros-Haydon spaces, which puts forward - as possibly the first step to solve the former problem - the question on the existence of a Gowers-Maurey type space $X$ with an infinite dimensional reflexive quotient algebra $\mathscr{L}(X)/\mathscr{SS}(X)$. The aim of this paper is to present a positive answer to this last question. 

\begin{theorem*}
    There is a Banach space $X$ with a basis so that $\mathscr{L}(X)$ is a direct sum of $\mathscr{SS}(X)$ and an infinite dimensional reflexive space $V$ with an unconditional basis $(I_s)_{s=0}^\infty$. The space $V$  as a Banach algebra is a unitisation of the space $\overline{\spa }\{I_s: s=1,2,\dots\}$ endowed with pointwise multiplication. 
\end{theorem*}

More precisely each $I_s$ is a projection on an infinite dimensional subspace $X_s\subset X$ spanned by some subsequence of the basis of $X$ and the basic sequence  $(I_s)_{s=1}^\infty$ is equivalent to the canonical basis of the dual space to a certain mixed Tsirelson space $Y$ defined in the sequel. Notice that in the presented approach in the place of shifts or diagonal operators used in the previous methods as operators generating $\mathscr{L}(X)$ modulo strictly singular operators we use projections on infinite dimensional subspaces of $X$. 

The construction of $X$ is based on Gowers-Maurey method of constructing HI (hereditary indecomposable) spaces, with two  adaptations. The first one   guarantees the existence of a sequence $(X_s)_{s=1}^\infty$ of totally incomparable 1-complemented Gowers-Maurey subspaces  of $X$, whose union is linearly dense in $X$. This step without the second ingredient would produce simply a Gowers-Maurey sum of spaces $(X_s)_s$, thus with $(X_s)_s$   forming an infinite dimensional Schauder decomposition of $X$ and the quotient algebra $\mathscr{L}(X)/\mathscr{SS}(X)$ isomorpic to the space $c_0$. 

The second ingredient added to the construction forces the equivalence of the sequence $(I_s)_{s=1}^\infty$ of projections on the subspaces $(X_s)_s$, to the canonical basis of the space $Y^*$, where $Y$ is a certain mixed Tsirelson space,
by adding to the norming set of $X$ a certain "ground" set  of functionals of $c_0$ type, that somewhat violate the "horizontal" (i.e. with respect to $(X_s)_s$) structure of $X$. This "ground" set in the norming set of $X$ guarantees a rich and closely related finite representation of $Y$ and $\ell_1$ in $X$. Precise choice of the parameters defining $X,Y$ yields the desired equivalence. 

Adding the second ingredient implies in particular that $(X_s)_{s=1}^\infty$ do not form a Schauder decomposition of $X$. Therefore justifying Gowers-Maurey "horizontal" behaviour of $X$, i.e. properties concerning sequences in $X$ that are block with respect to the sequence $(X_s)_{s=1}^\infty$, requires more technical work. The Gowers-Maurey "horizontal" properties of $X$ combined with Gowers-Maurey properties of all $X_s$, $s=1,2,\dots$, and bounded completeness of $(I_s)_{s=1}^\infty$  imply the desired representation of $\mathscr{L}(X)$ as the direct sum of $V$ and $\mathscr{SS}(X)$.

The author is very grateful to Pavlos Motakis for suggesting using $\ell_1$ instead of an originally introduced mixed Tsirelson space $Z$ as the second, after $Y$, auxiliary space, which essentially simplifies the reasonings. The author is also very grateful to the referee of the paper for thorough reading of the manuscript and suggesting several improvements of its presentation.

\ 

The paper is organised as follows. In Section 1. we recall the notion of mixed Tsirelson spaces and spaces of Gowers-Maurey type, as well as their basic properties. Section 2. is devoted to the definition of $X$ and its basic properties, including the key Lemma \ref{key}. Section 3. concentrates on the properties of  sequence of projections $(I_s)_{s=1}^\infty\subset\mathscr{L}(X)$. In Section 4. we prove the Gowers-Maurey type properties of the space $X$ adapting the classical reasoning with the aid of Lemma \ref{existence-g-separated}, and finally in Section 5. we examine the structure of $\mathscr{L}(X)$. The main theorem follows by Corollaries \ref{main1} and \ref{main2}.

We recall now some basic notation. 
By $c_{00}$ we denote the vector space of sequences in $\R$ which are eventually  zero, by $(e_n)_n$ and $(e_n^*)_n$ resp., we denote the unit vector sequence in $c_{00}$, treated  resp. as vectors or as functionals on $c_{00}$. The {support} of a $x=(a_i)_i\in c_{00}$ is the set $\supp (x)=\{i\in\N: a_i\neq 0\}$, the {range} $\rng(x)$ - the smallest interval in $\N$ containing the support of $x$. For any
$x=(a_i)_i\in c_{00}$ and $E\subset\N$ let $Ex$ be the vector defined by $(Ex)_i=a_i$ if $i\in E$ and $(Ex)_i=0$ otherwise. We write $n<x$ for $n\in\N$ and $x\in c_{00}$ if $n<\max\supp (x)$ and we write $x<y$ for vectors $x,y\in c_{00}$, if $\max\supp
(x)<\min \supp (y)$. A {block sequence } is any 
sequence $(x_n)_n\subset c_{00}$ satisfying $x_{1}<x_{2}<\dots$, a
{block subspace } of $c_{00}$ - any 
closed subspace spanned by an infinite block sequence. Given infinite $J\subset \N$ let $c_{00}(J)=\{x\in c_{00}: \supp (x)\subset J\}$ and $c_{00}^\mathbb{Q}$, resp. $c_{00}^\mathbb{Q}(J)$, be the $\mathbb{Q}$-vector space of elements of $c_{00}$, resp. $c_{00}(J)$, with coefficients in $\mathbb{Q}$. A subconvex combination of vectors $(x_i)_{i=1}^n\in c_{00}$ is any vector of the form $\sum_{i=1}^na_ix_i$ for some scalars $(a_i)_{i=1}^n\subset [0,1]$ with  $\sum_{i=1}^na_i\leq 1$. 

We say that a basic sequence $(e_n)_n$ $C$-dominates a basic sequence $(\tilde{e}_n)_n$ if the mapping $T:e_n\to \tilde{e}_n$, $n\in\N$, extends to a bounded operator on $\spa \{e_n: n\in\N\}$ with norm not exceeding $C$. If additionally $(\tilde{e}_n)_n$ $C$-dominates $(e_n)_n$ we say that these basic sequences are $C$-equivalent.

A {tree}  is a non-empty partially ordered  set $(\mt, \preceq)$ for which the set $\{ y \in \mt:y \preceq x \}$ is linearly ordered and finite for each $x \in \mt$.
The {root} is the smallest element of the tree (if it exists). The {terminal}  nodes are the maximal elements. A {branch} of $\mt$ is any maximal linearly ordered set in $\mt$.  The {immediate successors}  of $t \in \mt$ are all  the nodes $r \in \mt$ such that $t \prec r$ but there is no $s \in \mt$ with $t \prec s \prec r$, the set of immediate successors of $t$ is denoted by $succ (t)$.

\section{Mixed Tsirelson spaces and spaces of Gowers-Maurey type}

We recall the definition and basic properties of mixed Tsirelson spaces and  spaces built on their basis,  following \cite{ad,atod}. 

\begin{notation}\label{properties-integers}
For sequences $(n_j)_j, (m_j)_j\subset\N$ we consider the following properties:
\begin{itemize}
    \item[$(*)$] $m_1=2$, $n_1\geq 4$, $m_{j+1}\geq m_j^{m_j}$, $n_{j+1}\geq (10n_j)^{m_{j+1}^2}$ for all $j\in\N$,
    \item[$(**)$] $m_j,n_j\in\{2^k:k\in\N\}$ and for any $k\in\N$ there is $j\in\N$ with $n_{2j}=2^{km}$ for some $m\in\N$.
\end{itemize}
\end{notation}

Fix sequences $(\tilde{n}_j)_j, (\tilde{m}_j)_j$ satisfying $(*)$ of Notation \ref{properties-integers}. 

\begin{definition}[Mixed Tsirelson space]  Let $\tilde{K}\subset c_{00}$ be the smallest set satisfying the following:
\begin{enumerate}
    \item $(\pm e_n^*)_n\subset \tilde{K}$,
    \item for any $j\in\N$ and $f_1<\dots<f_d$ in $\tilde{K}$ with $d\leq \tilde{n}_j$ we have $\tilde{m}_j^{-1}(f_1+\dots+f_d)\in \tilde{K}$.
    \item $\tilde{K}$ is closed under restriction to intervals and taking subconvex combinations with rational coefficients. 
\end{enumerate}
Let $\|\cdot\|$ be the norm on $c_{00}$ with the norming set $\tilde{K}$, i.e. $\|\cdot\|=\sup\{|f(\cdot)|:f\in \tilde{K}\}$. The mixed Tsirelson space
$T[(\mathcal{A}_{\tilde{n}_j},\tilde{m}_j^{-1})_{j\in\N}]$ is the completion of $(c_{00}, \|\cdot\|)$. The set $\tilde{K}$ we shall call the canonical norming set of $T[(\mathcal{A}_{\tilde{n}_j},\tilde{m}_j^{-1})_{j\in\N}]$. 
\end{definition}

It is well-known that under the above assumptions the space $T[(\mathcal{A}_{\tilde{n}_j},\tilde{m}_j^{-1})_{j\in\N}]$ 
is a reflexive space with an unconditional basis formed by the unit vectors $(e_n)_n$ \cite[Thm I.10]{atod}. Moreover, the set $\tilde{K}$ is closed under restrictions to arbitrary subsets.

Any functional $f\in\tilde{K}$ produced as in (2) above, i.e. of the form $f=\tilde{m}_j^{-1}(f_1+\dots+f_d)$ is called weighted, with the weight defined as $\w(f)=\tilde{m}_j^{-1}$. We note that the weight of a functional in $\tilde{K}$ is not necessarily uniquely determined. 

\begin{definition}[The tree-analysis] Let $f\in \tilde{K}$. By a tree-analysis of
$f$ we mean a finite family $(f_t)_{t\in \mt}\subset\tilde{K}$ indexed by a finite tree $\mt$ with a
unique root $0\in \mt$ such that 
\begin{enumerate}
\item $f_{0}=f$,
\item $t\in \mt$ is terminal iff  $f_t\in (\pm e_n^*)$,
\item for every non-terminal $t\in \mt$ one of the following holds true
\begin{enumerate}
    \item $f_t=\tilde{m}_j^{-1}E\sum_{r\in \suc(t)}f_r$, for some $j\in\N$ with $\# \suc(t)\leq \tilde{n}_j$, an interval $E\subset \N$, and a sequence $(f_r)_{r\in \suc(t)}\subset \tilde{K}$ which is block with respect to an appropriate linear ordering of $\suc(t)$, 
    \item $f_t= \sum_{r\in \suc(t)}c_rf_r$ for some  scalars
 $(c_r)_r\subset [0,1]\cap\mathbb{Q}$ with $\sum_rc_r\leq 1$ and functionals $ (f_r)_{r\in \suc(t)}\subset \tilde{K}$. 
\end{enumerate}
\end{enumerate}
\end{definition}

We shall use also spaces of Gowers-Maurey type, introduced in \cite{gm}. We follow  the approach of \cite{atod}. We fix first an injective function $\tilde{\sigma}: \{(f_1,\dots,f_d)\subset c_{00}^\mathbb{Q}: f_1<\dots<f_d\}\to 2\N$ such that 
$\tilde{\sigma}((f_1,\dots,f_d))\geq \max\supp (f_d)\|f_1+\dots+f_d\|_\infty^{-1}$ for any non-zero $(f_1,\dots,f_d)$.
\begin{definition}[Mixed Tsirelson space]  Let $\tilde{K}_{\tilde{\sigma}}\subset c_{00}$ be the smallest set satisfying the following:
\begin{enumerate}
    \item $(\pm e_n^*)_n\subset \tilde{K}_{\tilde{\sigma}}$,
    \item for any $j\in\N$ and $f_1<\dots<f_d$ in $\tilde{K}_{\tilde{\sigma}}$ with $d\leq n_{2j}$ we have $\tilde{m}_{2j}^{-1}(f_1+\dots+f_d)\in \tilde{K}_{\tilde{\sigma}}$,
    \item for any $j\in\N$ and $j$-special sequence $f_1<\dots<f_d$ in $\tilde{K}_{\tilde{\sigma}}$ with $d\leq n_{2j-1}$ we have $\tilde{m}_{2j-1}^{-1}(f_1+\dots+f_d)\in \tilde{K}_{\tilde{\sigma}}$,
        \item $\tilde{K}_{\tilde{\sigma}}$ is symmetric, closed under restrictions to intervals and taking subconvex  combinations with rational coefficients. 
\end{enumerate}
In order to complete the definition of $\tilde{K}_{\tilde{\sigma}}$ we define special sequences as follows: $(f_1,\dots,f_d)\subset K_{\tilde{\sigma}}$ is $j$-special, $j\in \N$, if all $f_i$'s are weighted with $\w(f_1)=m_{2k}^{-1}$ for some $k\in\N$ with $m_{2k}>n_{2j-1}^2$, and $\w(f_{i+1})=m^{-1}_{\tilde{\sigma}((f_1,\dots,f_i))}$ for any $i=1,\dots, d-1$. 

Let $\|\cdot\|$ be the norm on $c_{00}$ with the norming set $\tilde{K}_{\tilde{\sigma}}$. The Gowers-Maurey  space
$T_{\tilde{\sigma}}[(\mathcal{A}_{\tilde{n}_j},\tilde{m}_j^{-1})_{j\in\N}]$ is the completion of $(c_{00}, \|\cdot\|)$. The set $\tilde{K}_{\tilde{\sigma}}$ is  called the canonical norming set of $T_{\tilde{\sigma}}[(\mathcal{A}_{\tilde{n}_j}, \tilde{m}_j^{-1})_{j\in\N}]$. 
\end{definition}
It follows that $T_{\tilde{\sigma}}[(\mathcal{A}_{\tilde{n}_j}, \tilde{m}_j^{-1})_{j\in\N}]$ is also a reflexive space with a basis formed by the unit vectors $(e_n)_n$. As in the case of $ \tilde{K}$, we use the notions of a weighted functional and a tree-analysis of a functional in $\tilde{K}_{\tilde{\sigma}}$.

We state now some of the standard notions and  results concerning Gowers-Maurey spaces in the form we shall need later.

\begin{definition}\label{def-exact}
    A pair $(f,x)$, with  $x\in T_{\tilde{\sigma}}[(\mathcal{A}_{\tilde{n}_j},\tilde{m}_j^{-1})_{j\in\N}]$, $f\in \tilde{K}_{\tilde{\sigma}}$, is called a $j$-exact pair, $j\in \N$, provided 
    \begin{enumerate}
        \item $f$ is weighted with $\w(f)=\tilde{m}_j^{-1}$,
        \item $f(x)=1$, $\|x\|\leq 3$, $\rng(x)=\rng(f)$,
        \item $|g(x)|\leq 5\max\{\tilde{m}_j^{-1},\w(g)\}$ for any weighted  $g\in\tilde{K}_{\tilde{\sigma}}$ with weight $\w(g)\neq \tilde{m}_j^{-1}$. 
    \end{enumerate}
\end{definition}
The next two lemmas follow by \cite[Prop. II.19 and Prop. II.25]{atod}. 

\begin{lemma}\label{ts-exact}
    Any block subspace of $T_{\tilde{\sigma}}[(\mathcal{A}_{\tilde{n}_j},\tilde{m}_j^{-1})_{j\in\N}]$ for any $j\in\N$ contains a vector $x$ so that $(x,f)$ is a $j$-exact pair for some $f\in\tilde{K}_{\tilde{\sigma}}$.
\end{lemma}

\begin{lemma}\label{ts-ris}
For any infinite $L\subset\N$ any block subspace of $T_{\tilde{\sigma}}[(\mathcal{A}_{\tilde{n}_j},\tilde{m}_j^{-1})_{j\in L}]$  contains a normalized block sequence $(y_n)_n$ so that for any $j\in\N\setminus L$ and $A\subset\N$ with $j<A$ and $\# A\leq \tilde{n}_j$ we have $\|\sum_{n\in A}y_n\|\leq 5\tilde{n}_j\tilde{m}_j^{-2}$.
\end{lemma}

Finally we recall the crucial property of Gowers-Maurey spaces. 

\begin{theorem}\label{xs}\cite[Thm IV.12]{atod}
Let $W=T_{\tilde{\sigma}}[(\mathcal{A}_{\tilde{n}_j},\tilde{m}_j^{-1})_{j\in \N}$ be a Gowers-Maurey space. Then any operator $T\in\mathscr{L}(W)$ is of the form $T=aId_{W}+S$ for some $a\in\R$ and $S\in\mathscr{SS}(W)$.   
\end{theorem}

\section{The definition and basic properties of the space $X$}

In this section we define of the space $X$, and present its basic properties, including  Lemma \ref{key}, which provides one of the tools substituting for the lack of "horizontal" Schauder decomposition of the space.

For the rest of the paper we fix sequences $(n_j)_j$, $(m_j)_j$ satisfying $(*)$ and $(**)$ of Notation \ref{properties-integers}.

\begin{notation}
We fix a partition $\N=\bigcup_{s=0}^\infty L_s$ into pairwise disjoint infinite sets with infinite $L_s\cap 2\N$ and $L_s\cap (2\N-1)$ for any $s=0,1,\dots$. 

Let 
$Y=T[(A_{n_{2j}},m_{2j}^{-1})_{j\in L_0}]$ with the canonical norming set $K_Y$.

We fix also another  partition $\N=\bigcup_{s=1}^\infty J_s$ of integers into pairwise disjoint infinite sets. 
Given any finite $A\subset \N$ we write $J_A=\bigcup_{s\in A}J_s$. \end{notation}

For any $x\in c_{00}$ let $\supp_h (x)=\{s\in\N: \supp (x)\cap J_s\neq\emptyset\}$ and let $\rng_h(x)$ be the minimal interval in $\N$ containing $\supp_h(x)$. A vector $x\in c_{00}$ is called simple, if  $\supp_h(x)$ is a singleton. A sequence $(x_n)_n\subset c_{00}$ is called horizontally block (h-block), provided it is block both with respect to the unit vector basis of $c_{00}$ and with respect to $(J_s)_{s\in\N}$, i.e. $(\supp (x_n))_n$ and $(\supp_h(x_n))_n$ are block in $\N$. A subspace of $c_{00}$ is called h-block, if it is generated by an infinite h-block sequence.

Now we set a family of mappings. For any $s=1,2,\dots$ set an injective mapping
$$
\sigma_s: \{(f_1,\dots,f_d)\subset c_{00}^\mathbb{Q}(J_s): (f_1,\dots,f_d) \text{ block}\}\to L_s\cap 2\N, \ \ \ s=1,2\dots
$$
such that $\sigma_s((f_1,\dots,f_d))>\max\supp (f_d)\|f_1+\dots+f_d\|_\infty^{-1}$ for any non-zero $(f_1,\dots,f_d)$.

Set a mapping
$$
\sigma:\{(f_1,\dots,f_d)\subset c_{00}^\mathbb{Q}: (f_1,\dots,f_d) \text{ h-block}\} \to 2^{L_0\cap 2\N}
$$
such that $\sigma((f_1,\dots,f_d))$ is infinite for any $(f_1,\dots,f_d)$ and $\sigma((f_1,\dots,f_d))\cap \sigma((g_1,\dots,g_l))=\emptyset$ for any $(f_1,\dots,f_d)\neq (g_1,\dots,g_l)$.  

Finally set a mapping 
$$
\tau:\{(f_1,\dots,f_d)\subset c_{00}^\mathbb{Q}: (f_1,\dots,f_d) \text{ h-block}, f_1,\dots,f_d \text{ simple}\} \to 2^\N
$$
such that for any $(f_1,\dots,f_d)$   and any $s\in\N$ we have $\tau(f_1,\dots,f_d)\cap (L_s\cap 2\N)\neq\emptyset$,  and $\tau(f_1,\dots,f_d)\cap \tau (g_1,\dots, g_l)=\emptyset$ for any  $(f_1,\dots,f_d)\neq (g_1,\dots,g_k)$.

\begin{definition}

We define the set $K\subset c_{00}$ as the minimal set such that

\begin{enumerate}
\item $\pm e_i^*\in K$ for any $i\in\mathbb{N}$,
\item for any $s\in\N$, $2j\in L_s$ and any block sequence $f_1<\dots<f_d$, of elements of $K\cap c_{00}(J_s)$, $d\leq n_{2j}$, 
we have $m_{2j}^{-1}(f_1+\dots+f_d)\in K$, 
\item for any $s\in\N$, $2j-1\in L_s$ and any $j$-special  (defined with respect to $\sigma_s$) 
sequence $f_1<\dots<f_d$ of elements of $K\cap c_{00}(J_s)$, $d\leq n_{2j-1}$, we have $m_{2j-1}^{-1}(f_1+\dots+f_d)\in K$, 
\item for any $2j\in L_0$ and any h-block sequence $f_1<\dots<f_d$ of elements of $K$, $d\leq n_{2j}$, we have 
$m_{2j}^{-1}(f_1+\dots+f_d)\in K$, 
\item for any $2j-1\in L_0$ and any $j$-h-special (defined with respect to $\sigma$) 
sequence $f_1<\dots<f_d$ of elements of $K$, $d\leq n_{2j-1}$, we have $m_{2j-1}^{-1}(f_1+\dots+f_d)\in K$, 
\item for any $N\in\N$ and any $N$-g-special sequence $f_1<\dots<f_{2N}$ of elements of $K$, and any $(\varepsilon_i)_{i=1}^N\subset\{-1,0,1\}$ we have $\sum_{i=1}^N\varepsilon_i(f_{2i-1}+f_{2i})\in K$,
\item $K$ is symmetric, closed under restrictions to intervals and taking subconvex combination with rational coefficients.
\end{enumerate}

In order to complete the definition we need to define special sequences. We say that a block sequence $(f_1,\dots,f_d)\subset K$, $1<d$, is 
\begin{itemize}
    \item $j$-special, $2j-1\in L_s$, $s\in\N$, provided $d\leq n_{2j-1}$, $\w(f_{i+1})=m_{\sigma_s(f_1,\dots, f_i)}^{-1}$ for each $i=1,\dots,d-1$ and $\w(f_1)< n_{2j-1}^{-2}$,
    \item $j$-h-special, $2j-1\in L_0$, provided $d\leq n_{2j-1}$, $(f_i)_{i=1}^d$ is h-block, $\w_h(f_{i+1})=m_{k}^{-1}$ for some  $k\in\sigma(f_1,\dots, f_i)$ for each $i=1,\dots,d-1$ and $\w_h(f_1)< n_{2j-1}^{-2}$ (cf. Notation \ref{notation3}),
    \item $N$-g-special, $N\in \N$, provided $d=2N$, all $f_i$'s are simple with $\supp_h(f_{2i-1})=\{i\}$, $\supp_h(f_{2i})=\{N+i\}$ for each $i\leq N$, $\w(f_{i+1})=m_k^{-1}$ for some $k\in\tau(f_1,\dots, f_i)$ for each $i=1,\dots,N-1$, and $\w(f_1)< N^{-1}$.

\end{itemize}

Let $\|\cdot\|$ be the norm on $c_{00}$ for which  $K$ is the norming set, i.e. $\|\cdot\|=\max\{|f(\cdot)|: f\in K\}$. Let $X$ be the completion of $(c_{00}, \|\cdot\|)$. 

\end{definition}

\begin{notation}\label{notation3}
A functional $f\in K$ (and its restrictions to an interval) obtained by (2) or (3) is called weighted with weight $\w(f)=m_j^{-1}$ for suitable $j$ (according to the standard terminology of mixed Tsirelson spaces). 

A functional $f$ obtained by (4) or (5) (or its restriction to an interval) is called h-regular and h-special resp., and in both cases h-weighted with weight $\w_h(f)=m_j^{-1}$ for suitable $j$. 

Finally a functional  obtained by (6) (or its restriction to an interval) is called g-special (if we underline the maximal length of the sequence used to generate the functional, we call it $N$-g-special for suitable $N\in\N$). 
\end{notation}

Notice that by the definition of $K$ the unit vectors $(e_i)_i$ form a bimonotone a basis of $X$.
We gather now basic facts concerning specific subspaces of $X$ encoded in the definition of the set $K$. 
\begin{notation}
For any fixed $s\in \N$ let     $X_s=\{x\in X: \supp_h(x)\subset J_s\}$ and let $K_s$ be the the  smallest subset of $K$ satisfying (1), (2), (3) and (7) of the definition of $K$ for  the fixed $s$. 
\end{notation}

\begin{remark}\label{xs-remark}
Fix $s\in\N$. As  $f|_{J_s}\in K_s$ for any $f\in K$, the space $X_s\subset X$ is a Gowers-Maurey space $T_{\sigma_s}[(\mathcal{A}_{n_j}, m_j^{-1})_{j\in L_s}]$ with the canonical norming set $K_s$. 

Moreover, by the definition of $K$, the projection $I_s:  X\to  X_s$, $s\in\N$, is bounded with norm 1. 

\end{remark}

\begin{notation}
In addition to the projections $(I_s)_{s=1}^\infty\subset\mathscr{L}(X)$ defined above, by $ I_0\in\mathscr{L}(X)$ we denote the identity operator $X\to X$. 

For any finite $A\subset\N$ we let also $I_A=\sum_{s\in A}I_s\in\mathscr{L}(X)$. 
\end{notation}

\begin{notation} By $\|\cdot\|_{h,\infty}$ we denote the supremum norm on $X$ with respect to $(X_s)_s$, namely for any $x\in X$ we let $\|x\|_{h,\infty}=\sup\{\|I_sx\|: s=1,2,\dots\}$. This norm is well-defined and satisfies $\|\cdot\|_{h,\infty}\leq \|\cdot\|$ by Remark \ref{xs-remark}. 
\end{notation}

As it was mentioned in the introduction, the definition of the norm of $X$ follows the scheme of mixed Tsirelson norm modelled on a certain "ground" set (cf. \cite{atod}), in our case formed by g-special functionals, in the "horizontal" setting. The latter provides a countable family of bounded non-trivial projections $(I_s)_{s=1}^\infty$ (on the spaces $(X_s)_{s=1}^\infty$), whereas including the "ground" set in the norming set ensures a rich and closely related finite representation of $Y$ and $\ell_1$ in $X$, in consequence - control over the norm of linear combinations of the aforementioned projections. Notice that the norming set without the ground set of g-special functionals would  define simply a Gowers-Maurey sum of Gowers-Maurey spaces $(X_s)_s$. The only "antihorizontal" construction appears in g-special functionals and it imposes the fact that $(X_s)_s$ do not form a Schauder decomposition of $X$.

We introduce now the basic tool for future estimates, i.e. a variant of a tree-analysis of a functional in $K$, here built up to simple functionals.

\begin{definition} 
An h-tree-analysis of
$f\in K$ is a family $(f_t)_{t\in \mt}$ indexed by a finite tree $\mt$ with a
unique root $0\in \mt$ such that the following hold

\begin{enumerate}
    \item $f_0=f$ and $f_t\in K$ for all $t\in \mt$,
    \item  $t\in \mt$ is  terminal iff  $f_t$ is simple,
    \item if $t\in \mt$ is non-terminal, then one of the following holds true:
    \begin{enumerate}
        \item $f_t$ is an h-weighted functional of the form $f_t= \pm\w_h(f_t)E\sum_{r\in \suc(t)}f_r$ for some interval $E\subset \N$ and a sequence  $(f_r)_{r\in \suc(t)}\subset K$ which is h-block with respect to an appropriate linear ordering of $\suc(t)$,
        \item $f_t$ is a g-special functional of the form $f_t=E\sum_{r\in \suc(t)}\varepsilon_rf_r$, for some interval $E\subset \N$, signs $(\varepsilon_r)_r\subset\{-1,0,1\}$ and a sequence  $(f_r)_{r\in \suc(t)}\subset K$ of simple functionals, which is h-block with respect to an appropriate linear ordering of $\suc(t)$, 
         \item $f_t$ is a subconvex rational combination of its successors, i.e. of the form $f_t=\sum_{r\in \suc(t)}c_rf_r$, for some  scalars $(c_r)_r\subset [0,1]\cap\mathbb{Q}$ with $\sum_rc_r\leq 1$, and a sequence $(f_r)_{r\in \suc(t)}\subset K$. 
    \end{enumerate}
\end{enumerate}
\end{definition}

Notice that every functional from a norming set $K$ admits an h-tree-analysis, not
necessarily unique.

The next two lemmas are crucial in the reasoning concerning horizontal structure of $X$. The second one forms a tool compensating for the fact that $(X_s)_{s=1}^\infty$ do not form a Schauder decomposition of $X$. 

\begin{lemma}\label{key-aux}
Fix $N\in\N$ and $f\in K$. Then there are $f', g\in K$ so that $f|_{J_{[1,N]}}=(f'+g)|_{J_{[1,N]}}$, $\supp_h (f')\subset [1,N]$ and $g$ is a subconvex combination of g-special functionals.

\end{lemma}

\begin{proof}
Take an h-tree-analysis $(f_t)_{t\in\mathcal{T}}$ of $f$, assume that $\supp_h(f)\cap [1,N]\neq\emptyset$.  Let $A\subset \mt$ be the family of (terminal)
$r\in\mt$ such that $r\in\suc(t)$ for some $t\in\mt$ with $f_t$ g-special and $\min \supp_h(f_t)\leq N<\max\supp_h(f_t)$. 
Let $B$ be the family of all terminal $t\in\mt\setminus A$ with $\min\supp_h(f_t)\leq N$. Let $\mt_A\subset\mt$ (resp. $\mt_B$) be the family of all nodes in $\mt$ that are comparable with some node of $A$ (resp. $B$). Note that the root of $\mt$ belongs to $\mt_A\cup\mt_B$.

We define $f'$ and $g$ simultaneously by defining their h-tree-analyses  $(g_t)_{t\in\mt_A}$ and $(f'_t)_{t\in\mt_B}$  
inductively starting from terminal nodes, so that 
\begin{enumerate}
    \item $f_t|_{J_{[1,N]}}=(f'_t+g_t)|_{J_{[1,N]}}$ for any $t\in\mt_B\cup \mt_A$ (we adopt the convention that $g_t=0$ if $t\not\in \mt_A$ and $f'_t=0$ if $t\not\in\mt_B$),
    \item for any $t\in\mt_A$, either $t\in\suc (s)$ for some $s\in\mt$ with $g_s$ g-special, or $g_t$ is a subconvex combination of g-special functionals,
    \item  $\supp_h(f'_t)\subset [1,N]$ for any $t\in\mt_B$, 
    \item $f_t'=f_t$ for any $t\in\mt_B$ with $\supp_h(f_t)\leq N$,
    \item $f'_t$ is h-weighted with $\w_h(f'_t)=\w_h(f_t)$ for any $t\in\mt_A\cup\mt_B$ with $f_t$ h-weighted. 
\end{enumerate}

If $t\in\mt_A\cup\mt_B$ is terminal  let $g_t=f_t$ for $t\in\mt_A$ and $f'_t=f_t$ for $t\in\mt_B$.

Pick non-terminal $t\in\mt_A\cup \mt_B$ and assume all $f'_r$ and $g_r$, $r\in\suc(t)$, are  defined and satisfy the above conditions. 

If $f_t$ is g-special, then $t$ belongs to only one of the trees $\mt_A, \mt_B$. Let $f_t'=f_t$ if $t\in\mt_B$ and
$g_t=f_t$ if $t\in \mt_A$. All desired conditions are satisfied. 

If $f_t$ is a subconvex combination of its successors, let $\suc_B(t)=\suc(t)\cap\mt_B$ and $\suc_A(t)=\suc (t)\cap \mt_A$. 
If $f_t=\sum_{r\in \suc(t)}c_rf_r$ then let $f'_t=\sum_{r\in \suc_B(t)}c_rf'_r$ and $g_t=\sum_{r\in\suc_A(t)}c_rg_r$. Notice  that $f_t|_{J_{[1,N]}}=\sum_{r\in \suc_A(t)\cup \suc_B(t)}c_rf_r|_{J_{[1,N]}}$, thus (1) is satisfied by the inductive assumption, (3) is satisfied by the inductive assumption. For (4) notice that if $\supp_h(f_t)\leq N$, then also $\supp_h(f_r)\leq \N$ for any $r\in\suc(t)$ and apply the inductive assumption. Notice also that in the considered case, by the inductive assumption (2), for any $r\in\suc_A(t)$ the functional $g_r$ is a subconvex combination of g-special functionals. Therefore we 
obtain (2) for $t$. It is clear also that $f'_t,g_t\in K$. 

If $f_t= \pm\w_h(f_t)E\sum_{r=1}^df_r$, with $f_1<\dots<f_d$, let 
$r_0=\max\{r=1,\dots,d: \supp_h(f_r)\leq \N\}$. Note that by inductive assumption $f'_r=f_r$ for any $1\leq r\leq r_0$. 

If $r_0=d$ or $r_0+1\not\in \mt_A\cup \mt_B$ then $\supp_h(f_t)\cap [1,N]\subset \bigcup_{r=1}^{r_0}\supp_h(f_r)$ and we let $g_t=0$ and 
$f'_t= \pm\w_h(f_t)E\sum_{r=1}^{r_0}f'_r$.  Conditions (1) are (5) are satisfied  by the definition, (3) and (4) are satisfied by the inductive assumption. It follows easily that $f'_t\in K$. 

Assume $r_0+1\in\mt_A\cup\mt_B$. Then  $\supp_h(f_t)\cap [1,N]\subset \bigcup_{r=1}^{r_0+1}\supp_h(f_r)$ and we let $g_t=\pm\w_h(f_t)Eg_{r_0+1}\in K$ and 
$f'_t=\pm \w_h(t)E\sum_{r=1}^{r_0+1}f'_r$. 
Conditions (1), (2) and (5) are satisfied  by the definition, (3) and (4) - by the inductive assumption. Notice also that $f'_t\in K$ both in the case of $f_t$ h-regular and h-special by the inductive assumption (4) and (5) for $r\in\{1,\dots,r_0+1\}$. This case ends the inductive construction and thus the proof of the lemma. 
\end{proof}

\begin{lemma}\label{key}
\begin{enumerate}
    \item For any normalized h-block sequence $(x_n)_n$ there are $\delta>0$, an infinite $J\subset\N$ and an h-block sequence  $(f_n)_{n\in J}\subset K$ such that $f_n(x_n)\geq \delta$ for any $n\in J$. 
\item For any normalized h-block sequences $(x_n)_n$ and $(y_n)_n$ with $\inf_n\dist(x_n,\R y_n)>0$ and $\|y_n\|_{h,\infty}\xrightarrow{n\to\infty} 0$ there are $\delta>0$, infinite $J\subset\N$ and an h-block sequence $(f_n)_{n\in J}\subset K$ such that $f_n(x_n)\geq \delta$ for any $n\in J$ and $f_n(y_n)\xrightarrow{J\ni n\to\infty}0$. 
\end{enumerate}
\end{lemma}

\begin{proof}
We shall prove (2), as the proof of (1) is a simpler version of (2). Take $(x_n)_n$ and $(y_n)_n$ as above. Notice that by rational convexity the set $K$ is $w^*$-dense in the unit ball of $X^*$. Therefore by Hahn-Banach theorem we can pick $(\hat{f}_n)_n\subset K$ with $\hat{f}_n(y_n)<(2n)^{-1}$ and $\hat{f}_n(x_n)>\epsilon$ for each $n\in\N$ and some universal $\epsilon >0$. Relabelling we can assume that $\supp_h(x_n)\cup \supp_h (y_n)>n$ and $\|y_n\|_{h,\infty}\leq (2n)^{-2}$ for all $n$. 

We introduce some useful notation: for any $N$-g-special $g=\sum_{i=1}^N\varepsilon_i(f_{2i-1}+f_{2i})$ and any $ k\in\N$ we write $g^k=\sum_{i=1}^{\min\{k,N\}}\varepsilon_i(f_{2i-1}+f_{2i})$. Obviously $g^k\in K$ and $g-g^k\in K$. If $g\in K$ is a subconvex combinations of g-special functionals, say $g=\sum_rc_rg_r$, then we write $g^k=\sum_rc_rg_r^k$. 

For any $n\in\N$ pick $f'_n, g_n\in K$ by Lemma \ref{key-aux} for $N=n$ and $f=\hat{f}_n$. Let $h_n^k=\hat{f}_n-f'_n-g_n^k$ for each $k,n\in\N$. Then for all $k,n\in\N$ we have the following
\begin{enumerate}
    \item $\frac{1}{3}h_n^k\in K$ by definition of $K$,
    \item $\supp_h(h_n^k)>\min\{k,n\}$ as $\hat{f}_n|_{[1,\min\{k,n\}]}=(f'_n+g_n^k)|_{[1,\min\{k,n\}]}$,
    \item $h^k_n(x_n)=\hat{f}_n(x_n)-g_n^k(x_n)$ and $h^k_n(y_n)=\hat{f}_n(y_n)-g_n^k(y_n)$ as $\supp_h(f'_n)\leq n<\supp_h(x_n)\cup \supp_h(y_n)$,
    \item if additionally $k\leq n$, then we have $|h_n^k(y_n)|=|g_n^k(y_n)|+(2n)^{-1}\leq2k\|y_n\|_{h,\infty}+(2n)^{-1} <n^{-1}$ by assumption on $(y_n)_n$, (3) and definition of g-special functionals.
\end{enumerate}

Assume that for any $k\in\N$ there is $n_k> k$, with 
$|h^k_{n_k}(x_{n_k})|>\epsilon/2$. 
Then any h-block subsequence  $(\frac{1}{3}h^k_{n_k})_{k\in J}$ (which exists by (2) above) satisfies the desired properties for $(y_{n_k})_{k\in J}$ and $(x_{n_k})_{k\in J}$ by (4).

Otherwise for some $k_0\in\N$ we have $|g_n^{k_0}(x_n)|>\epsilon/2$ for any $n>k_0$ by (3) above and the choice of $(\hat{f}_n)_n$.

Fix now  $n>k_0$. We shall restrict $g_n^{k_0}$ to some $g'_n\in K$ with better properties. 

First for any $N$-g-special $g=\sum_{i=1}^N\varepsilon_i(f_{2i-1}+f_{2i})$ let $g'=g-g^l$, where $l=\max\{n+1-N,0\}$. Then in particular for any g-special $g$ we have $\supp_h((g^{k_0})')\cap (k_0,\infty)=\supp_h((g^{k_0})')\cap (n,\infty)=\supp_h(g^{k_0})\cap (n,\infty)$. If 
$g_n^{k_0}=\sum_rc_rg_{r,n}^{k_0}$ let $g'_n=\sum_rc_r(g_{r,n}^{k_0})'\in K$ and notice that also 
\begin{enumerate}
    \item[(5)] $\supp_h(g'_n)\cap (k_0,\infty)=\supp_h(g'_n)\cap (n,\infty)=\supp_h(g^{k_0}_n)\cap (n,\infty)$, 
    \item[(6)]
    $g'_n(x_n)=g^{k_0}_n(x_n)$ by (5),
    \item[(7)] moreover $\#(\supp_h(g'_n)\cap (k_0,\infty))\leq k_0$ by definition of $(g_n^{k_0})$.  
\end{enumerate}
Let 
$f_n=k_0^{-1}g'_n|_{J_{(k_0,\infty)}}$ for all $n>k_0$. Then $f_n\in K$ (as a convex combination of a family $(g'_n|_{J_s})_{s>n, s\in \supp_h(g'_n)}$, by (7)), $\supp_h(f_n)\geq n$ by (5), $f_n(x_n)>\epsilon/2k_0$ by (6) and $|f_n(y_n)|\leq \#(\supp (g'_n)\cap (n,\infty))\|y_n\|_{h,\infty}\leq 1/n$ for all $n>k_0$ by (7), thus passing to a suitable
subsequence ends the proof. 
\end{proof}

\section{The sequence $(I_s)_{s=1}^\infty\subset\mathscr{L}(X)$}

This section is devoted to the proof of  Thm \ref{main0} which reduces the study of $(I_s)_s$ to the case of $(e^*_s)_s\subset Y^*$. The reasoning presented here relies on the properties of classical Tsirelson spaces and the specific choice of the space $Y$. 

We introduce first the notion of g-special pairs of sequences of vectors and functionals, with the functional parts being a g-special sequences. By the careful choice of g-special functionals included in the norming set $K$ of $X$ the vector parts of a g-special pair provide closely related finite representation of $Y$ and $\ell_1$ in $X$ (cf. Prop. \ref{crucial} (1) and (2)). This representation is rich enough to guarantee that vectors from isomorphic copies of finite dimensional subspaces of $Y$ are "seminormalizing" for any $T\in \spa  \{I_s: s=0,1,\dots\}$ (cf. Prop. \ref{crucial} (3)) which provides the control on the norm of such $T$ described in Thm \ref{main0}.  
\begin{definition}
    A pair of sequences $((x_1,\dots,x_{2N}),(f_1,\dots,f_{2N}))$, with $(x_s)_s\subset X$, $(f_s)_s\subset K$, is called a $N$-g-special pair provided
    \begin{enumerate}
        \item $(f_1,\dots,f_{2N})$ is an $N$-g-special sequence,
        \item $\supp_h(x_s)=\supp_h(f_s)=\{r_s\}$ for suitable $r_s\in\N$, for any $s=1,\dots,2N$,
        \item $(f_s,x_s)$ is a $2j_s$-exact pair in $X_{r_s}$, where $\w(f_s)=m_{2j_s}^{-1}$ for any  $s=1,\dots,2N$ (cf. Def. \ref{def-exact}).
    \end{enumerate}
\end{definition}
In particular any vector part $(x_s)_{s=1}^{2N}$ of a g-special pair is a block sequence with $1\leq \|x_s\|\leq 3$ for each $s$ (by definition of exact pairs). 

\begin{proposition}\label{crucial}
\begin{enumerate}

\item For any $N$-g-special pair $((x_1,\dots, x_{2N}), (f_1,\dots, f_{2N}))$ the sequence $(x_{2s-1})_{s=1}^N$ 1-dominates the unit vector basis $(e_s)_{s=1}^N\subset \ell_1$ and is 3-dominated by $(e_s)_{s=1}^N\subset \ell_1$.

\item For any $N\in\N$ there is $M_N\in \N$ so that for any $N$-g-special pair $((x_1,\dots, x_{2N}), (f_1,\dots, f_{2N}))$ with $\supp(x_1)>M_N$  the sequence $(x_{2s-1}-x_{2s})_{s=1}^N$ $1$-dominates the unit vector basis $(e_s)_{s=1}^N\subset Y$ and is $20$-dominated by $(e_s)_{s=1}^N\subset Y$.

\item For any scalars $(a_s)_{s=1}^M$ and any $N$-g-special pair $((x_1,\dots,x_{2N}), (f_1,\dots,f_{2N}))$ with $M\leq N$, $\supp(x_1)>M_N$, there are scalars $(b_s)_{s=1}^M$ 
so that the vector $x=\sum_{s=1}^Mb_s(x_{2s-1}-x_{2s})\in X$ satisfies 
$\|x\|\leq 1$ and $\|\sum_{s=1}^Ma_sI_s\|\leq 40\|\sum_{s=1}^Ma_sI_sx\|$.
\end{enumerate}
\end{proposition}
\begin{proof}

(1) As $\|x_s\|\leq 3$ for any $s$, the sequence $(x_{2s-1})_{s=1}^N$ is 3-dominated by $(e_s)_{s=1}^N\subset \ell_1$. 
For the reverse domination fix scalars $(b_s)_{s=1}^N$. By definition $f=\sum_{s=1}^N\sgn (b_s)(f_{2s-1}+f_{2s})$ is a g-special functional in $K$. Thus 
$$
\|\sum_{s=1}^Nb_sx_{2s-1}\|\geq f(\sum_{s=1}b_sx_{2s-1})=\sum_{s=1}|b_s|
$$

\

(2) First notice that $(x_{2s-1}-x_{2s})_{s=1}^N$ 1-dominates $(e_s)_{s=1}^N\subset Y$. Indeed, pick scalars $(b_s)_s$ and $\tilde{f}\in K_Y$ with $\tilde{f}_t(\sum b_se_s)=\|\sum b_se_s\|_Y$. Write $\tilde{f}=\sum_sd_se_s^*$, then, as $K$ contains all h-regular functionals, also $f=\sum_sd_sf_{2s-1}\in K$ and thus
$$
\|\sum_sb_s(x_{2s-1}-x_{2s})\|\geq f(\sum_sb_s(x_{2s-1}-x_{2s}))=\sum_s d_sb_s=\|\sum_sb_se_s\|_Y
$$

For the reverse domination 
fix $N\in\N$ and let $M_N=\sup\{\supp(\bfg): \bfg \text{ an h-special sequence}, \sigma(\bfg)\cap [1,12N]\neq \emptyset\}$. Notice that $M_N$ is finite by injectivity of $\sigma$. It follows that for any h-special sequence $(g_1,\dots,g_d)\subset K$,  $1\leq n< d$ with $\max\supp(g_n)>M_N$ and an interval $E$ we have $\|E(g_{n+1}+\dots+g_d)\|_{h,\infty}\leq \w_h(g_{n+1})\leq 1/12N$.

We shall show that  for any  g-special pair
$((x_1,\dots, x_{2N}), (f_1,\dots, f_{2N}))$ with $\supp (x_1)>M_N$ and any scalars $(b_s)_{s=1}^N$ with $\|(b_s)_s\|_\infty\leq 1$ we have $\|\sum_{s=1}^Nb_s(x_{2s-1}-x_{2s})\|\leq 10\|(b_s)_{s=1}^N\|_{Y}+\frac{1}{2}$. 
Having this for any $(b_s)_{s=1}^N$ with $\|\sum_{s=1}^Nb_s(x_{2s-1}-x_{2s})\|=1$ we have $\|(b_s)_{s=1}^N\|_{Y}\geq \frac{1}{20}$, which ends the proof. 

Pick $(b_s)_{s=1}^N$ with $\|(b_s)_s\|_\infty\leq 1$, and a g-special pair $((x_1,\dots, x_{2N}), (f_1,\dots, f_{2N}))$ with $\supp(x_1)>M_N$. Let $x=\sum_{s=1}^Nb_s(x_{2s-1}-x_{2s})$, then $\|x\|_{h,\infty}\leq \max\{\|x_s\|: s=1,\dots,2N\}\leq 3$ by def. of g-special pairs and exact pairs. 

We notice first the following two estimates.

By the definition of $M_N$ for any h-special functional $g=\pm\w_h(g)E\sum_{n=1}^dg_n$ and   $n_0=\min \{n: \max\supp( g_n)>M_N\}$ we have the following estimate 
\begin{align}\label{one-bullet}
|g(x)|
&\leq \w_h(g)|Eg_{n_0}(x)|+\sum_{n>n_0}\sum_{s=1}^{2N}|Eg_n(x_s)|\notag\\
&\leq \w_h(g)|Eg_{n_0}(x)|+\Bigl\|E\sum_{n>n_0}g_n\Bigr\|_{h,\infty}\cdot\max_s\|x_s\|\cdot\#\Bigl\{s=1,\dots,2N:\Bigl(E\sum_{n>n_0}g_n\Bigr)(x_s)\neq 0\Bigr\}\notag\\
&\leq \w_h(g)|Eg_{n_0}(x)|+ (4N)^{-1}\#([1,2N]\cap \supp_h(E\sum_{n>n_0}g_n)) \end{align}

Take any g-special functional $g=E\sum_{s=1}^L\varepsilon_s(g_{2s-1}+g_{2s})$. For any $s=1,\dots,2L$ put $s'=\lceil \frac12 s\rceil$. 
Let $s_0=\max\{s=1,\dots,2L: g_s=f_s \}$ if the latter set is non-empty, otherwise let $s_0=0$. Pick also $s_1$ with $\min E\in\supp(g_{s_1})$. We shall consider only the case of odd $s_0\geq s_1>0$ with $s_0,s_1\in\supp_h(g)$, as in other cases the computation is even simpler. 
By the definition of g-special pairs and exact pairs  we have
\begin{align*}
|g(x)|&=
    |E\sum_{s=1}^L \varepsilon_sb_s (g_{2s-1}(x_{2s-1})-g_{2s}(x_{2s}))|
    \\
    & \leq |Eg_{s_1}(b_{s'_1}x_{s_1})|+|Eg_{s_0}(b_{s_0'}x_{s_0})|+|Eg_{s_0+1}(b_{(s_0+1)'}x_{s_0+1})|+\sum_{s_0+1< s\in\supp_h(g)}|Eg_s(b_{s'}x_s)|
    \\
    &\leq |b_{s'_1}|\|x_{s_1}\|+|b_{s_0'}|\|x_{s_0}\|+|b_{(s_0+1)'}|\|x_{s_0+1}\|+\sum_{s_0+1< s\in\supp_h(g)}|b_{s'}|\max\{\w(g_s), \w(f_s)\}
    \\
    & \leq 9\|(b_{s'})_{s\in \supp_h(g)}\|_\infty+5/4^{s_0+1}\|(b_{s'})_{s\in \supp_h(g)}\|_\infty
\end{align*}
Therefore we have the following
\begin{equation}\label{two-bullet}
    |g(x)|\leq 10\|(b_{s'})_{s\in \supp_h(g)}\|_\infty
\end{equation}

Take any functional $g\in K$ and its h-tree-analysis $(g_t)_{t\in\mt}$. We shall produce a functional $\tilde{g}\in K_{Y}$ so that $|g(x)|\leq 10\tilde{g}((b_s)_{s=1}^N)+\frac{1}{2}$ by constructing its tree-analysis $(\tilde{g}_t)_{t\in\tilde{\mt}}$. 

We prove inductively, starting from the terminal nodes, that for any  node $t\in\mt$ there is a functional $\tilde{g}_t\in K_Y$ with $\supp (\tilde{g}_t)\subset\supp_h(g_t)$ so that $|g_t(x)|\leq 10\tilde{g}_t((b_s)_{s=1}^N)+\delta_t$, where $\delta_t=(4N)^{-1}\#([1,2N]\cap \supp_h(g_t))$. 

If $t\in\mt$ is terminal, then we let $\tilde{g}_t=\sgn (b_{s_t})e^*_{s_t}$ where $\supp_h(g_t)=\{s_t\}$ obtaining $|g_t(x)|\leq 3|b_{s_t}|=3\tilde{g}_t((b_s)_s)$.

Take non-terminal  $t\in\mt$  and assume that all $(\tilde{g}_r)_{r\in\suc(t)}$ are  defined and satisfy the desired condition.

If $g_t$ is g-special, then all of successors of $t$ are terminal, we pick $s_0\in\supp_h(g_t)$ with $|b_{s'_0}|=\|(b_{s'})_{s\in \supp_h(g_t)}\|_\infty$,  let $\tilde{g}_t=\sgn(b_{s'_0})e^*_{s_0}$ 
and obtain $|g_t(x)|\leq 10\tilde{g}_t((b_s)_s)$ by \eqref{two-bullet}. 

If $g_t$ is a subconvex combination of $(g_r)_{r\in \suc(t)}$,  then we pick $r_0\in \suc(t)$ with the biggest $10|\tilde{g}_{r_0}((b_s)_s)|+\delta_{r_0}$ and let $\tilde{g}_t=\tilde{g}_{r_0}$. Then $\supp(\tilde{g}_t)\subset\supp_h(g_t)$ and $|g_t(x)|\leq 10\tilde{g}_t((b_s)_s)+\delta_t$ by the inductive assumption. 

If $g_t$ is h-regular, $g_t=\pm\w_h(g_t)E\sum_{r\in \suc(t)}g_r$, then let $\tilde{g}_t=\pm\w_h(g_t)\supp_h(g_t)\sum_{r\in \suc(t)}\tilde{g}_r$ and notice $\tilde{g}_t\in K_Y$ by the inductive assumption and the definition of h-weighted functionals in $K$. As  $\sum_{r\in \suc(t)}\delta_r\leq \delta_t$ by inductive assumption  we have the desired estimate. 

If $g_t$ is h-special  we pick $n_0$ as in  \eqref{one-bullet} and let $\tilde{g}_t=\supp_h(g_t)\tilde{g}_{n_0}$. The desired estimate follows by the inductive assumption and \eqref{one-bullet}, which finishes the inductive construction. 

Therefore we have $g(x)\leq 10\|(b_s)_s\|_Y+\delta_g\leq 10\|(b_s)_s\|_Y+\frac{1}{2}$, which ends the proof.

\ 

(3) Let $T=\sum_{s=1}^Ma_sI_s$, pick a block $w\in X$ with $\|w\|=1$ and $\|T\|\leq 2\|Tw\|$. Take $f\in K$ with its h-tree-analysis $(f_t)_{t\in \mt}$ such that $f(Tw)=\|Tw\|$. 
For any $s\in \supp_h(f)$ let $b_s=f_{t_s}(w)$, where $t_s\in\mt$ with $f_{t_s}\in K_s$ and let $b_s=0$ otherwise. 

We claim that $\|(b_s)_s\|_Y\leq 2\|w\|=2$. 
First for any $t\in\mt$ with $f_t$ g-special of the form $f_t=E\sum_{r=1}^L\varepsilon_r(f_{2r-1}+f_{2r})$ define $A_t=\supp_h(E\sum_{r=1}^L\varepsilon_rf_{2r-1})$ and $B_t=\supp_h(E\sum_{r=1}^L\varepsilon_rf_{2r})$. Let $A=\bigcup_{t\in\mt: f_t \text{ g-special}}A_t$, $B=\supp_h(f)\setminus A$ and notice that for any $s,\bar{s}\in\N$, if $s,\bar{s}\in A$ or $s,\bar{s}\in B$, then we have $s<\bar{s}$ iff $f_{t_s}<f_{t_{\bar{s}}}$. 

Take $\tilde{f}\in K_Y$ with $\tilde{f}((b_s)_s)=\|(b_s)\|_Y$ of the form $\tilde{f}=\sum_{s\in \supp_h(f)} d_s e_s^*$ (by unconditionality of the basis of $Y$). By the above property of $A$ and $B$, as $K$ contains all h-regular functionals, we have $f_A=\sum_{s\in A}d_sf_{t_s}\in K$ and $f_B=\sum_{s\in B}d_sf_{t_s}\in K$. Therefore 
$$
\tilde{f}((b_s)_s)=\sum_{s\in\supp_h(f)}d_sf_{t_s}(w)=f_A(w)+f_B(w)\leq 2\|w\|
$$
Consider the vector $x=\sum_{s=1}^Mb_s(x_{2s-1}-x_{2s})$. Then by (2) $\|x\|\leq 20\|(b_s)_{s=1}^M\|_{Y}\leq 40$. 
On the other hand by (1) 
$$
\|Tw\|=\sum_{s=1}^Ma_sf_{t_s}(w)\leq  \|\sum_{s=1}^Ma_sb_sx_{2s-1}\|=\|Tx\|
 $$
which proves that  $x/40$ is the desired vector.
\end{proof}
    
\begin{theorem}\label{main0}

\begin{enumerate}
    \item  For any $T\in \spa \{I_s: s=1,2,\dots\}$ we have $\|T\|\leq 2^7\dist(T,\mathscr{SS}(X))$.
    \item The sequences $(I_s)_{s=1}^\infty\subset\mathscr{L}(X)$ and $(e^*_s)_s\subset Y^*$ are $2^7$-equivalent. 
\end{enumerate}
\end{theorem}

\begin{proof}

(1) Fix scalars $(a_s)_{s=1}^N$, let $T=\sum_{s=1}^Na_sI_s$ and fix $S\in\mathscr{SS}(X)$. 
By the strict singularity of $S$, Remark \ref{xs-remark} and Lemma \ref{ts-exact} pick inductively on $s=1,\dots,2N$ a g-special pair $((x_1,\dots,x_{2N}), (f_1,\dots,f_{2N}))$ with $\supp(x_1)>M_N$ and $\|Sx_s\|\leq \|T\|/2^{s+6}$ for each $s$. 
By Prop. \ref{crucial} (3) for some scalars $(b_s)_{s=1}^N\subset [-1,1]$  the vector $x=\sum_{s=1}^Nb_s(x_{2s-1}-x_{2s})$ has norm at most 1 and 
$\|T\|\leq 40\|Tx\|$. 
Moreover, by the choice of $(x_s)$ we have $\|Sx\|\leq \|T\|/2^6$, thus $\|T-S\|\geq \|(T-S)x\|\geq\|T\|/2^7$ and we finish the proof of (1).

(2) We shall prove that for any $(a_s)_{s=1}^N$ we have
$$
\frac{1}{120}\|\sum_{s=1}^Na_sI_s\|\leq \|\sum_{s=1}^Na_se^*_s\|_{Y^*}\leq 20\|\sum_{s=1}^Na_sI_s\|
$$

Let $T=\sum_{s=1}^Na_sI_s$.

In order to show the first inequality pick by Remark \ref{xs-remark} and Lemma \ref{ts-exact} any g-exact pair $((x_1,\dots,x_{2N}), (f_1,\dots, f_{2N}))$ with $\supp(x_1)>M_N$. By 
Prop. \ref{crucial} (3) and (2) for some scalars $(b_s)_{s=1}^N$ with $\|(b_s)_{s=1}^N\|_Y\leq 1$ the vector $x=\sum_{s=1}^Nb_s(x_{2s-1}-x_{2s})$ satisfies  $\|T\|\leq 40\|Tx\|$. Thus 
$$
\|T\|\leq 40\|Tx\|=40
\|\sum_{s=1}^Na_sb_sx_{2s-1}\|\leq  120\sum_{s=1}^N|a_sb_s|\leq 120\|\sum_{s=1}a_se^*_s\|_{Y^*}
$$
as $\|(b_s)_{s=1}^N\|_Y\leq 1$ and the basis $(e_n)_n\subset Y$ is 1-unconditional.

For the second inequality pick $y=(b_s)_{s=1}^N\in Y$ of norm 1 on which $\sum_{s=1}^Na_se_s^*$ attains its norm, by Remark \ref{xs-remark} and Lemma \ref{ts-exact} pick any g-special pair $((x_1, \dots, x_{2N}), (f_1,\dots,f_{2N}))$ with $\supp(x_1)>M_N$ and let $x=\sum_{s=1}^Nb_s(x_{2s-1}-x_{2s})$. By Prop. \ref{crucial} (2) $\|x\|\leq 20$, whereas by Prop. \ref{crucial} (1) we have
$$
\|\sum_{s=1}^Na_se^*_s\|_{Y^*}=\sum_{s=1}^Na_sb_s\leq \|\sum_{s=1}^Na_sb_sx_{2s-1}\|=\|Tx\|
$$
\end{proof}

As $dist(I_0,\overline{\spa }\{I_s: s=1,2,\dots\})=1$ and the space $Y$ is reflexive, we have the following

\begin{corollary}\label{main1}
The  sequence $(I_s)_{s=0}^\infty\subset\mathscr{L}(X)$ is an unconditional shrinking and boundedly complete  basic sequence.

\end{corollary}

\section{The  Gowers-Maurey type notions and properties in the space $ X$}

We repeat in this section  the standard scheme of Gowers-Maurey spaces theory as presented in \cite{atod} in the setting the horizontal structure of $X$, presenting in detail the proofs requiring some adjustment due to "antihorizontal" behaviour of g-special functionals and quoting the rest. The adjustments rely on Lemma \ref{key} and the following notion, essential in the proof of the Basic Inequality (Lemma \ref{basic-inequality}) for certain special h-block sequences.

\begin{definition}[g-separated sequences]
    An h-block sequence $(x_n)_{n=1}^N\subset X$ is called $(C,\epsilon)$-g-separated, $C>1$ and $\epsilon>0$, if there is some increasing sequence $(K_n)_{n=1}^N\subset\N$ such that 
    \begin{enumerate}
        \item $\|x_n\|\leq C$, for any $n=1,\dots,N$
        \item $f(x_n)=0$ for any $K$-g-special $f$ with $K\leq K_n$, for any $n=1,\dots,N$ 
        \item $|f(x_n)|<\epsilon$ for any $K$-g-special $f$ with $K\geq K_{n+1}$, for any $n=1,\dots,N$. 
    \end{enumerate}
\end{definition}

\begin{lemma}\label{existence-g-separated}
 For any $C>1$, $\epsilon>0$ and  $(C,\epsilon)$-g-separated $(x_n)_{n=1}^N\subset X$ there is $K_{N+1}\in\N$ so that for any $x\in X$ with $\|x\|\leq C$ and $\supp_h(x)>2K_{N+1}$ the sequence $(x_1, \dots, x_N, x)$ is $(C,\epsilon)$-g-separated. 
 
 In particular any bounded h-block sequence $(x_n)_n\subset X$ for any $\epsilon>0$ has a $(C,\epsilon)$-g-separated subsequence, where $C=\sup_n\|x_n\|$. 
\end{lemma}
\begin{proof} Consider $(x_n)_{n=1}^N$ as above, with associated $(K_n)_{n=1}^N$. Pick
$K_{N+1}>K_N$ with \\
$C\max\supp(x_N)/K_{N+1}<\epsilon$. Notice that for any $K$-g-special $f$ with $K\geq K_{N+1}$ we have  $|f(x_N)|\leq C\max\supp(x_N)\|f\|_\infty\leq C\max\supp(x_N)/K_{N+1}<\epsilon$ and we obtain (3) for $x_N$. Pick $x\in X$ with $\supp_h(x)>2K_{N+1}$ and notice that for any $K$-g-special $f$ with  $K\leq K_{N+1}$ we have $\supp_h(f)\leq 2K_{N+1}<\supp_h(x)$ and hence (2) is satisfied for $x$, which ends the proof. 
\end{proof}

\begin{definition}[$\ell_1$-averages]
    A normalized vector $x\in X$ is called a $C$-$\ell_1$-average of length $N\in\N$, $C>1$, if there is an h-block sequence $(x_n)_{n=1}^N\subset X$ with $x=N^{-1}(x_1+\dots+x_N)$ and $\|x_n\|\leq C$ for each $n=1,\dots,N$.  
\end{definition}
If $C=2$ we will omit it in the notation and say  that $x$ is an $\ell_1$-average of suitable length. 

\begin{lemma}\label{ell-averages-existence}
    Any h-block subspace of $X$ contains an $\ell_1$-average of arbitrary length. 
\end{lemma}
\begin{proof}
We follow here the proof of \cite[Lemma 1]{gm}. Fix $k\in\N$ and  a normalized h-block sequence $(x_n)_n\subset X$. Passing to a subsequence if necessary by Lemma \ref{key} (1) we can assume that there is an h-block sequence $(f_n)_n\subset K$ and $\delta>0$ such that $f_n(x_n)>\delta$ and $f_n(x_m)=0$ for any $n\neq m$. 
By property $(**)$ of sequences $(n_j)_j$, $(m_j)_j$ (cf. Notation \ref{properties-integers}) pick $j\in\N$ so that $n_{2j}=2^{km}$, for $m$ big enough to ensure $m_{2j}<\delta\sqrt[k]{n_{2j}}=\delta2^m$ (the latter is possible by $(*)$, as $m_j/\sqrt[k]{n_j}\to 0$). For every $0\leq i\leq m$ and $1\leq j\leq 2^{k(m-i)}$ let $x(i,j)=\sum_{n=(j-1)2^{ki}+1}^{j2^{ki}}x_n$. Assume towards contradiction that for $1\leq i\leq m$ and $1\leq j\leq 2^{k(m-i)}$,  $2^{-ki}2^{-(i-1)}x(i,j)$ is not an $\ell_1$-average of length $2^k$. Notice that for any $0\leq i< m$,  $x(i+1,j)$ is a sum of $2^k$ many terms of the form $x(i,l)$. It follows by induction that $\|x(i,j)\|\leq 2^{-i}2^{ki}$, in particular $\|x(m,1)\|\leq 2^{-m}2^{km}$. On the other hand 
$$
\|x(m,1)\|=\|\sum_{n=1}^{2^{km}}x_n\|\geq m_{2j}^{-1}\sum_{n=1}^{2^{km}}f_n(x_n)\geq n_{2j}m_{2j}^{-1}\delta=2^{km}m_{2j}^{-1}\delta 
$$ 
which contradicts the choice of $m$. Therefore we proved that there exist $\ell_1$-averages of length $2^k$ for any $k\in\N$ in the subspace spanned by $(x_n)_n$, thus also $\ell_1$-averages of arbitrary length. Indeed, for $l\in\N$ pick $k\in\N$ with $l\leq 2^k$ and an $\ell_1$-average $x=2^{-k}\sum_{i=1}^{2^k}y_i$. Let $\mathcal{F}=\{F\subset\{1,\dots,2^k\}, \#F=l\}$ and write
$$
x=l2^{-k}\binom{2^k-1}{l-1}^{-1}\sum_{F\in\mathcal{F}}y_F
$$
where $y_F=l^{-1}\sum_{i\in F}y_i$ for each $F\subset\{1,\dots,2^k\}$. Thus $x=(\# \mathcal{F})^{-1}\sum_{F\in\mathcal{F}}y_F$ and so there is at least one $F\in\mathcal{F}$ with $\|y_F\|=1$, i.e. $y_F$ is an $\ell_1$-average of length $l$. 
\end{proof}

The next Lemma is standard, cf. \cite[Lemma II.23]{atod}
\begin{lemma}\label{av-est} Let $x\in X$ be an $\ell_1$-average of length $N$. Then for any intervals $E_1<\dots<E_d$ with $d\leq N$ we have
$ \sum_{i=1}^d\|I_{E_i}x\|\leq 6$. 
\end{lemma}

We modify below the classical definition of RIS by adding the assumption on g-separation in (4).
\begin{definition}[Rapidly Increasing Sequences] An h-block sequence $(x_k)\subset
X$ is called a $(C,\epsilon)$- {rapidly increasing sequence} (RIS), for $C>1$, $\epsilon>0$, if there is a strictly increasing
sequence $(j_k)_k\subset L_0$ such that 

\begin{enumerate}

\item $\|x_k\|\leq C$, for any $k$,

\item $2m_{j_1}^{-1}<\epsilon$ and $2m_{j_{k+1}}^{-1}\max\supp_h(x_k)<\epsilon$ for any $k$,

\item $| f(x_k)|\leq C\w_h(f)$ for any h-weighted $f\in K$ with $\w_h(f)> m_{j_k}^{-1}$ for any $k$,

\item $(x_k)_k$ is $(C,\epsilon)$-g-separated.
\end{enumerate}

\end{definition}

\begin{remark}\label{ris-existence} Let $(x_k)$ be an h-block sequence of $\ell_1$-averages
of strictly increasing lengths $(n_{j_k})\subset\N$. 
Then by  Lemmas \ref{existence-g-separated} and \ref{av-est} there is infinite $J\subset\N$ so that $(x_k)_{k\in J}$ is a $(6,\epsilon)$-RIS. 
\end{remark}

The next Lemma is the key tool, classical by now, for estimating the norms of averages of RIS, reducing it to computing norms of averages of the unit basic vectors in the auxiliary mixed Tsirelson
space $W=T[(\mathcal{A}_{2n_j},\frac{1}{m_j})_{j\in L_0}]$ with the canonical norming set $K_{W}$.

\begin{lemma}[Basic Inequality] \label{basic-inequality} Let $(x_k)_{k=1}^N\subset X$ be a $(C,\epsilon)$-RIS for some $C>1$ and $\epsilon>0$.

Then for any norming functional $f\in K$ there is a norming functional $g\in K_{W}$ satisfying
$$
| f(\sum_{k=1}^Nx_k)|\leq Cg(\sum_{k=1}^Ne_k)+C\epsilon N
$$
and such that if $f$ is h-weighted, then either $g$ is also weighted with $\w(g)=\w_h(f)$ or $g$ is some unit vector $e_t^*$.

Moreover, if additionally for a fixed $j_0\in\N$ and for any h-weighted $h\in K$ with $\w_h(h)=m_{j_0}^{-1}$
and any interval $E$ we have $| h(\sum_{k\in E}x_k)|\leq C$, then the functional
$g$ can be chosen to have a tree-analysis $(g_t)_{t\in \mathcal{T}}$ with $\w(g_t)\neq m_{j_0}$
for any $t\in \mathcal{T}$ with $g_t$ weighted. \end{lemma}

\begin{proof} Let $(m_{j_k})_k$ be the sequence of weights associated to the RIS $(x_k)_k$. We repeat the standard reasoning, with the only difference in CASE 5 below.

Take $f\in K$ with a tree-analysis $(f_t)_{t\in \mathcal{T}}$. We
will prove the Lemma under the additional assumption for some
fixed $j_0$. The proof in general case follows along the same
lines, just without referring to the case of weight $m_{j_0}$.
We prove by induction, starting from the terminal nodes, that for
any $t\in \mathcal{T}$ we have the following: for any interval $E\subset\N$ there is functional $g_t\in K_{W}$ with $\supp g_t\subset E$  such that
$$
| f_t(\sum_{k\in E}x_k)|\leq Cg_t(\sum_{k\in E}e_k)+C\epsilon\# E
$$
and if $f_t$ is h-weighted, then either $g_t$ is weighted with $\w(g_t)=\w_h(f_t)$ or $g_t$ is a unit vector $e_{n_t}^*$. Moreover, if $g_t$ is weighted, then $\w(g)\neq m_{j_0}^{-1}$.

If $t\in\mathcal{T}$ be terminal, then $f_t$ is simple with $\supp_h(f_t)=\{s_t\}$ for some
$s_t\in\N$, thus letting $g_t=e_{\min E}^*$ we finish the case. 

Take a non-terminal node $t\in \mathcal{T}$ and
interval $E$ and assume we proved the assertion for all its successors. Consider the following cases.

CASE 1. $f_t$ is h-weighted with $\w_h(f_t)=m_{j_0}^{-1}$, then by 
the additional assumption
$$
| f_t(\sum_{k\in E}x_k)|\leq C\leq Ce^*_{\min
E}(\sum_{k\in E}e_k)
$$
and we let $g_t=e^*_{\min E}$.

CASE 2. $f_t$ is h-weighted with $\w_h(f_t)\neq m_{j_0}^{-1}$ and there is
some $k\in E$ such that $m_{j_k}\leq \w_h(f_t)^{-1}$. Let $k_t$ be maximal
in $E$ with this property. By the condition (2) of RIS we have
$$
| f_t(\sum_{k\in E, k<k_t}x_k)|\leq \w_h(f_t)\max
\supp_h( x_{k_t-1}) \leq\frac{\epsilon}{2}
$$
and by conditions (2) and (3) of RIS
$$
| f_t(\sum_{k\in E, k> k_t}x_k)|\leq
C\w_h(f_t)\# E\leq \frac{C\epsilon}{2}\# E
$$
Hence for $g_t=e^*_{k_t}$ we have the desired estimate.

CASE 3. $f_t$ is h-weighted with $\w_h(f_t)\neq m_{j_0}^{-1}$ and for all $k\in E$ we have $m_{j_k}>\w_h(f_t)^{-1}$. Let
$$
E_0=\{k\in E: \ \rng_h( x_k) \text{ intersects horizontal ranges of at least two different }f_s, s\in \suc( t) \}
$$
$$
E_s=\{k\in E\setminus E_0: \ \rng_h (x_k) \text{ intersects }\rng_h(f_s)\}, \ \ \  \ s\in \suc( t)
$$
Since each $E_s$ is an interval, by inductive assumption we pick for each $f_s$ and $E_s$ a suitable $g_s\in K_W$ with $\supp (g_s)\subset E_s$. Notice  $\# E_0\leq \#\suc(t)$ and thus
$\{E_s:\ s\in \suc(t)\}, \{\{k\}:\ k\in E_0\}$ is a family of pairwise disjoint intervals of cardinality at most $ 2\# \suc(t)$. Therefore $g_t\in K_{W}$ and $\supp (g_t)\subset E$, where
$$
g_t=\w_h(f_t)(\sum_{s\in \suc(t)}g_s+\sum_{k\in E_0}e^*_k)
$$
Estimate by the condition (2) of RIS and the inductive assumption
\begin{align*}
| f_t(\sum_{k\in E}x_k)| & \leq
\w_h(f_t)\sum_{s\in \suc(t)}| f_s(\sum_{k\in E_s}x_k)| + \sum_{k\in E_0}| f_t(x_k)| 
\\
& \leq C\w_h(f_t)\sum_{s\in \suc(t)}g_s(\sum_{k\in E_s}e_k)+C\epsilon\sum_{s\in \suc(t)}\# E_s+C\w_h(f_t)\# E_0\\
& \leq Cg_t(\sum_{k\in E}e_k)+C\epsilon\# E
\end{align*}

CASE 4. $f_t$ is a subconvex rational
combination of functionals: $f_t=\sum_{s\in \suc(t)}c_sf_s$. By the
inductive assumption we pick for any $s\in \suc(t)$ and $E$
a suitable $g_s\in W$, let $g_t=\sum_{s\in \suc(t)}c_sg_s$ and
compute
\begin{align*}
| f_t(\sum_{k\in E}x_k)| & \leq \sum_{s\in \suc(t)}c_s| f_s(\sum_{k\in E}x_k)|\leq C\sum_{s\in \suc(t)}c_sg_s(\sum_{k\in E}e_k)+C\epsilon\# E\sum_{s\in \suc(t)}c_s\\
& =Cg_t(\sum_{k\in E}e_k)+C\epsilon\# E
\end{align*}

CASE 5. $f_t$ is $K$-g-special for some $K\in\N$. Let $(K_n)_n\subset\N$ be the sequence associated to $(x_n)$ as a g-separated sequence. Pick $n_0$ with $K_{n_0}< K\leq K_{n_0+1}$. Then $f_t(x_n)=0$ for any $n> n_0$ and $|f_t(x_n)|\leq \epsilon$ for any $n<n_0$. Let $g_t=e^*_{\min E}$, then 
$$
|f_t(\sum_{k\in E}x_k)|\leq Cg_t(\sum_{k\in E}e_k)+\sum_{k<n_0, k\in E}\epsilon\leq Cg_t(\sum_{k\in E}e_k)+\epsilon \# E
$$
which ends the proof. 
\end{proof}

\begin{proposition}\label{basic-inequality-cor} Fix $j\in\N$. Let $(x_k)_{k=1}^{n_j}\subset X$ be a
$(C,\frac{1}{m_j^2})$-RIS. Then for any h-weighted functional $f\in K$  we have
$$
| f(\frac{1}{n_j}\sum_{k=1}^{n_j}x_k)|\leq \begin{cases} \frac{2C}{m_j^2} \ \ \ \ \ \ \ \ \text{if}\ \ \ \w_h(f)<m_j^{-1} \\
\frac{2C}{m_j} \ \ \ \ \ \ \ \ \text{if}\ \ \ \w_h(f)=m_j^{-1}\\
\frac{3C\w_h(f)}{m_j} \ \ \ \text{if}\ \ \ \w_h(f)>m_j^{-1}
\end{cases}
$$
In particular
$$
\|\frac{1}{n_j}\sum_{k=1}^{n_j}x_k\|\leq \frac{2C}{m_j}
$$
Moreover, if the additional assumption for $j_0=j$ in the Basic Inequality (Lemma
\ref{basic-inequality}) is satisfied, then
$$
\|\frac{1}{n_j}\sum_{k=1}^{n_j}x_k\|\leq \frac{4C}{m_j^2}
$$
\end{proposition}

\begin{proof}The proof is a straightforward application of estimates on the basis of $W$ (cf. \cite[Lemma II.9]{atod})  and Lemma \ref{basic-inequality}. \end{proof}

The next definition follows along the lines of  the classical one with slight changes required in our setting, which do not affect the results. 

\begin{definition}[Dependent sequences] A double sequence $(x_k,f_k)_{k=1}^{n_{2j-1}}\subset
X\times K$ is called a $(\eta,D,j)$-{dependent sequence}, for
$\eta\in\{0,1\}$, $D\geq 1$, $2j-1\in L_0$, if

\begin{enumerate}

\item $(f_k)_k$ is a $j$-h-special sequence,

\item $f_k(x_k)=1$ if $\eta=1$ and $|f_k(x_k)|\leq 2^{-k}$ if $\eta=0$ and $\rng_h(x_{k-1})<\rng_h (f_k)<\rng_h (x_{k+1})$ for any $k$,

\item $2m_{j_1}^{-1}<m_{2j-1}^{-2}$ and $2m_{j_{k+1}}^{-1}\max\supp_h(x_k)<m_{2j-1}^{-2}$ for any $k$.

\item $x_k=d_k\frac{m_{j_k}}{n_{j_k}}\sum_{i=1}^{n_{j_k}}y_i^k$ for some
$(6,\frac{1}{m_{j_k}^2})$-RIS $(y^k_i)_{i=1}^{n_{j_k}}$ where $m_{j_k}^{-1}=\w_h(f_k)$ and $d_k\in [1,D]$ for any $k$,

\item $(x_k)_k$ is a $(12D,\frac{1}{m_{2j-1}^2})$-g-separated sequence,

\end{enumerate}

\end{definition}

\begin{remark}\label{obstacle}
    Notice that the condition (3)  of dependent sequences, redundant in the classical situation, is required in order to prove that the vector part of a dependent sequence is itself a suitable RIS, as we cannot manage $\rng_h(x_k)=\rng_h(f_k)$ (since projections $(I_{[1,N]})_N$ are not uniformly bounded). This condition is also the reason for defining $\sigma$ as a multivalued function, which guarantees  existence of dependent sequences. 
\end{remark}

\begin{proposition}\label{dep-est} Let $(x_k,f_k)_{k=1}^{n_{2j-1}}$ be a $(\eta,D,j)$-dependent
sequence, $D\geq 1$, $2j-1\in L_0$. Then
\begin{align*}
\|\frac{1}{n_{2j-1}}\sum_{k=1}^{n_{2j-1}}(-1)^kx_k\|\leq \frac{180D}{m_{2j-1}^2}
& \ \ \ \ \ \text{if}\ \ \ \eta=1 \\
\|\frac{1}{n_{2j-1}}\sum_{k=1}^{n_{2j-1}}x_k\|\leq \frac{180D}{m_{2j-1}^2} & \ \
\ \ \ \text{if} \ \ \  \eta=0
\end{align*}

\end{proposition}

\begin{proof}
Notice that the sequence $(x_k)_{k=1}^{n_{2j-1}}$ is itself
a $(18D,\frac{1}{m_{2j-1}^2})$-RIS  with $\|x_k\|\leq 12D$ for any $k$ by Prop.  \ref{basic-inequality-cor} applied to the $(6,\frac{1}{m_{j_k}^2})$-RIS
$(y_i^k)_{i=1}^{n_{j_k}}$ for each $k$ and conditions (3) and (5). 
Thus the same hold for 
$((-1)^kx_k)_{k=1}^{n_{2j-1}}$ also.

Let $\eta=1$. By Prop.  \ref{basic-inequality-cor} it is enough to show that for any $j$-h-special $h\in K$ and any interval $E$ we have $| h(\sum_{k\in E}(-1)^kx_k)|\leq 45D$. Take $j$-h-special $h\in K$  and an interval $E\subset\N$. Then $h$ is of the form $h=\frac{1}{m_{2j-1}}(Af_l+\dots+f_r+h_{r+1}+\dots+h_d)$ for some $j$-h-special sequence
$(f_1,\dots,f_r,h_{r+1},\dots,h_{n_{2j-1}})$, $1\leq l\leq r\leq
d\leq n_{2j-1}$, $A\subset\N$ interval, where $r$ is minimal with $h_{r+1}\neq
f_{r+1}$.

For any $k,i$ with $\w_h(h_i)\neq m_{j_k}$ by Prop. \ref{basic-inequality-cor} and the definition of dependent sequences we have $|
h_i(x_k)|\leq \frac{18D}{n_{2j-1}^2}$. Now assume that $\min E\leq r$ and compute
\begin{align*}
| h(\sum_{k\in E}(-1)^kx_k)| & \leq
|(Af_l+\dots+f_r)((-1)^{\min E}x_{\min
E}+\dots+(-1)^rx_r)|
\\
&+|
h_{r+1}(x_{r+1})|+\sum_{w(h_i)\neq
m_{j_k}} | h_i(x_k)|\\
& \leq 12D+1+12D+n_{2j-1}^2\frac{18D}{n_{2j-1}^2}\leq 43D
\end{align*}
In other cases of $E$ the computation is even simpler. The case $\eta=0$ follows analogously. \end{proof}

\begin{lemma}\label{dep-ex} For any h-block subspaces $U,V$ of $X$ there is some $D\geq 1$ so that for any $2j-1\in L_0$ there is a $(1, D, j)$-dependent sequence
$(x_k,f_k)_{k=1}^{n_{2j-1}}$ with
$(x_{2k})_{k=1}^{n_{2j-1}/2}\subset U$ and
$(x_{2k-1})_{k=1}^{n_{2j-1}/2}\subset V$. \end{lemma}

\begin{proof} By Lemma \ref{existence-g-separated} 
pick h-block sequences $(y_{2i})_i\subset U$, $(y_{2i-1})_i\subset V$ of $\ell_1$-averages of increasing length. By Lemma \ref{key},  passing to subsequences and allowing some perturbation, we can assume that for some h-block sequences $(f_{2i})_i, (f_{2i-1})_i\subset K$ and $0<\delta_1, \delta_2\leq 1$ we have $f_{2i-1}(y_{2i-1})=\delta_1$  and $f_{2i}(y_{2i})=\delta_2$. 
 
Now fix $j\in L_0$ and pick inductively on $k$, using Remark \ref{ris-existence}, Lemma \ref{existence-g-separated}, the definition of $\sigma$ and the fact that $(y_{2i-1})_i$, $(y_{2i})_i$, $(f_{2i-1})_i$, $(f_{2i})_i$ are h-block, 
$(6,\frac{1}{m_{j_k}^2})$-RIS $(y_i^k)_{i=1}^{n_{j_k}}\subset (y_{2i})_i$ and $(f_i^k)_{i=1}^{n_{j_k}}\subset (f_{2i})_i$
for $k\in 2\N$ and $(6,\frac{1}{m_{j_k}^2})$-RIS $(y_i^k)_{i=1}^{n_{j_k}}\subset (y_{2i-1})_i$ and $(f_i^k)_{i=1}^{n_{j_k}}\subset (f_{2i-1})_i$
for $k\in 2\N-1$ such that for $f_k=m_{j_k}^{-1}\sum_{i=1}^{n_{j_k}}f_i^k$, $k=1,\dots,n_{2j-1}$, $d_{2k}=\delta_2^{-1}$, $d_{2k-1}=\delta_1^{-1}$ and $x_k=d_k\frac{m_{j_k}}{n_{j_k}}\sum_{i=1}^{n_{j_k}}y_i^k$, $k=1,\dots,n_{2j-1}$  conditions (1)-(4) of dependent sequences are satisfied. In particular with multivalued $\sigma$ we can guarantee (3). Passing to a further sequence by Lemma \ref{existence-g-separated} we obtain the condition (5).

Then the double sequence $(x_k,f_k)_{k=1}^{n_{2j-1}}$, is a 
$(1,D,j)$-dependent sequence, with $D=\max\{\frac{1}{\delta_1}, \frac{1}{\delta_2}\}$, with desired properties. \end{proof}

Now we are ready to prove the horizontal HI  property of the space $X$. 

\begin{theorem}\label{hi} No two h-block subspaces of $X$ form a direct sum.   
\end{theorem}

\begin{proof} Fix h-block subspaces  $U,V$ of $X$.  Pick $D\geq 1$ ensured by Lemma \ref{dep-ex} and for any $2j-1\in L_0$ choose a $(1,D,j)$-dependent sequence
$(x_k,f_k)_{k=1}^{n_{2j-1}}$ with $x_{2k}\in U$ and $x_{2k-1}\in
V$ for any $k=1,\dots,n_{2j-1}/2$. Let
$$
y=\frac{m_{2j-1}}{n_{2j-1}}\sum_{k=1}^{n_{2j-1}/2}x_{2k}\in
U, \ \ \ \
z=\frac{m_{2j-1}}{n_{2j-1}}\sum_{k=1}^{n_{2j-1}/2}x_{2k-1}\in V
$$
Then by Proposition \ref{dep-est} we have $\|y-z\|\leq \frac{180D}{m_{2j-1}}$,
whereas
$$
\|y+z\|\geq \frac{1}{m_{2j-1}}\sum_{k=1}^{n_{2j-1}}f_k(y+z)\geq 1
$$
which, as $j$ is arbitrary big, ends the proof. 
\end{proof}

\section{Bounded operators on $X$}

In this section we prove the "scalar+horizontal strictly singular" property of $X$ continuing the Gowers-Maurey scheme from the previous section and describe  the behaviour of bounded operators on $X$ with respect to subspaces $X_s$, $s\in\N$. The reasoning, based on classical mixed Tsirelson spaces tools including the results from the previous section and on already shown complete boundedness of $(I_s)_s$, leads to Cor. \ref{main2}. The latter combined  with Cor. \ref{main1} proves the main theorem.

\begin{proposition}\label{op-xs} Fix $s\in\N$.
  \begin{enumerate}
      \item The spaces $X_s$ and $(I-I_s)(X)$ are totally incomparable (i.e. no  infinite dimensional subspaces $U\subset X_s$ and $V\subset (I-I_s)(X)$ are isomorphic). 
      \item   For any $T\in\mathscr{L}( X)$ there is $a_s\in\R$  with  $T\circ I_s-a_s I_s\in\mathscr{SS}(X)$.
 \end{enumerate}
 \end{proposition}
 
 \begin{proof}
(1) Assume the contrary. Let  $T: U\to V$, for some infinite dimensional subspaces $U\subset X_s$, $V\subset (I-I_s)(X)$, be an isomorphism. Allowing some perturbation we can pick a normalized block sequence $(u_n)_n\subset U$ with $(Tu_n)_n\subset V$ also block. Let $U'=\overline{\spa }\{u_n: n\in\N\}$. 

Assume first that for some $s\neq t\in\N$ the operator $I_t\circ T: U'\to X_t$ is not strictly singular. Passing to a further block sequence we can assume that $I_t\circ T: U'\to X_t$ is an isomorphism onto its image. 
Then by Remark \ref{xs-remark} and  Lemma \ref{ts-ris} for any $2j\in L_t$ pick a normalized block sequence  $(z_i)_{i=1}^{n_{2j}}$ of $(u_n)_n$ with $\|z_1+\dots+z_{n_{2j}}\|\leq 5n_{2j}m_{2j}^{-2}$. 
On the other hand, by the lower estimate in $X_t$, as $(Tz_i)_i$ is also a block sequence, we have $\|I_tTz_1+\dots+I_tTz_{n_{2j}}\|\geq n_{2j}m_{2j}^{-1}\|T\|^{-1}$, which for large $2j\in L_t$ yields a contradiction. 

As all operators $I_t\circ T$, $t\neq s$, are strictly singular, allowing some perturbation pick a  normalized block sequence $(u'_n)_n\subset U'$ with $(Tu'_n)_n$ (seminormalized) h-block. By Remark \ref{xs-remark} and  Lemma \ref{ts-ris} for any $2j\in L_0$  pick a  normalized block sequence $(w_i)_i$ of $(u'_n)_n$ with $\|\sum_{i\in A}w_i\|\leq 5n_{2j}m_{2j}^{-2}$ provided $\# A\leq n_{2j}$. 
By Lemma \ref{key} (1) we obtain h-block $(f_i)_{i\in J}\subset K$ with $f_i(Tw_i)\geq \delta$ for some universal $\delta>0$ and all $i\in J$ and $f_i(Tw_k)=0$ for any $i\neq k\in J$. Therefore for any $2j\in L_0$ and $A\subset J$, $\# A\leq n_{2j}$, we have 
$$
\|\sum_{i\in A}Tw_i\|\geq m_{2j}^{-1}\sum_{i=1}^{n_{2j}}f_i(Tw_i)\geq n_{2j}m_{2j}^{-1}\delta
$$
which by the choice of $(w_i)_i$ contradicts the boundedness of $T$.  

(2) For any bounded operator $T\in\mathscr{L}(X)$ write  $T\circ I_s=(I-I_s)\circ T\circ I_s+ I_s\circ T\circ I_s$. Then apply (1) and the  "scalar+strictly singular" property of $X_s$ (Remark \ref{xs-remark} and Thm \ref{xs}). 
\end{proof}

We shall use the horizontal version of the strictly singular property. 

\begin{definition}
An operator $Q\in\mathscr{L}( X)$ is horizontally strictly singular if for no h-block subspace $U$ of $X$ the operator $Q|_U: U\to X$ is an isomorphism onto its image.   \end{definition}

\begin{proposition} \label{op-ss}
Let $S\in\mathscr{L}( X)$ satisfy $S\circ I_s\in\mathscr{SS}(X)$ for any $s\in\N$. Then $S\in\mathscr{SS}(X)$.   \end{proposition}

\begin{proof} 
By the assumption on $S$ pick a normalized h-block sequence  $(x_s)_s\subset X$, so that  $\supp_h(x_s)=\{s\}$ and $\|Sx_s\|< 4^{-s}$ for each $s\in\N$. It follows that for any normalized h-block sequence $(y_n)_n$ of $(x_s)_s$ we have $Sy_n\to 0$. By Thm \ref{hi} the operator $S$ is horizontally strictly singular. 

Assume that there is a block subspace $U$ of $X$ such that $S|_U$ is an isomorphism onto its image. 

Fix $s\in\N$. Notice there is $\varepsilon_0>0$ and $k_0\in\N$ such that for any normalized $u\in U$ with $\supp(u)>k_0$ we have $\|(I-I_s)(u)\|\geq\varepsilon_0$. Otherwise we could pick a block sequence $(u_n)_n\subset U$ with $u_n-I_su_n\to 0$, which contradicts the strict singularity of $S\circ I_s$. 

It follows that for any $\epsilon>0$ and $k\in\N$ there is a normalized vector $u\in U$ with $\supp(u)>k$ and $\|I_su\|<\epsilon$. Otherwise there would be a finite-codimensional subspace $U'$ of $U$ with spaces $I_s(U')$, $(I-I_s)(U')$ isomorphic to $U'$, which contradicts (1) of Prop. \ref{op-xs}. 

As $s\in\N$ is arbitrary, with to some perturbation we can pick a normalized h-block sequence $(u_n)_n\subset U$, which contradicts the horizontal strict singularity of $S$. 
\end{proof}

Continuing the reasoning from the previous section we prove the "scalar+horizontal strictly singular" property of $X$. Recall that $I_0$ denotes the identity on $X$. 

\begin{theorem}\label{scalar+hss}
    Any bounded operator $T\in\mathscr{L}( X)$ is  of the form $a_0 I_0+Q$, where $Q$ is horizontally strictly singular. 
\end{theorem}
\begin{proof} 
By Prop. \ref{op-xs} (1) for any $s\leq r\in\N$ the operator $(I_1+\dots+I_s)\circ T\circ (I-(I_1+\dots+I_r))$ is strictly singular, thus for any  $s\leq r\in\N$ and any $\epsilon>0$ there is a block $z\in X$ with $\|z\|=1$, $\supp_h(z)>r$ and $\|(I_1+\dots+I_s)(Tz)\|<\epsilon$. Applying this observation inductively and allowing some perturbation we pick a normalized h-block sequence $(z_i)_i$ such that $(Tz_i)_i$ is also h-block.

By Lemma \ref{ell-averages-existence} pick an h-block sequence $(y_i)_i$ of $\ell_1$-averages of increasing length in the block subspace spanned by $(z_i)_i$. We claim that $\dist
(Ty_i,\R y_i)\xrightarrow{i\to\infty} 0$. 

Assume the contrary. Then  $\inf_{i\in J}\dist (Ty_i, \R y_i)>0$ for some  infinite $J\subset\N$. By Lemma \ref{key} (2), passing to a subsequence if necessary and allowing some perturbation we can assume that there is h-block $(f_n)_n\subset K$ and $\delta>0$ with  $f_n(Ty_n)>\delta $ and $|f_n(y_n)|\leq 2^{-n}$ for all $n$. 

As in the proof of Lemma \ref{dep-ex}, using Remark \ref{ris-existence},  Lemma \ref{existence-g-separated} and the definition of $\sigma$, produce for any $j\in\N$ a $(0,1,j)$-dependent sequence
$(x_k,f_k)_{k=1}^{n_{2j-1}}$, picking for each $k=1,\dots,n_{2j-1}$ suitable
$(y_i^k)_{i=1}^{n_{j_k}}$ and $(f_i^k)_{i=1}^{n_{j_k}}$ from $(y_i)$ and $(f_i)$ and
letting $x_k=\frac{m_{j_k}}{n_{j_k}}\sum_{i=1}^{n_{j_k}}y_i^k$ and
$f_k=\frac{1}{m_{j_k}}\sum_{i=1}^{n_{j_k}}f_i^k$ (notice the choice here is simpler, as the condition $|f_k(x_k)|\leq 2^{-k}$  is automatic, therefore $d_k=1$ for all $k$ and thus $D=1$). For such $x_k,f_k$ we have
$f_k(Tx_k)>\delta$ for each $k$. Let $x=\frac{1}{n_{2j-1}}\sum_{k=1}^{n_{2j-1}}x_k$ and
$f=\frac{1}{m_{2j-1}}\sum_{k=1}^{n_{2j-1}}f_k$. Then $\|x\|\leq \frac{180}{m_{2j-1}^2}$ by Prop. \ref{dep-est}. On the other hand, since $f\in K$, we
have $\|Tx\|\geq f(Tx)>\frac{\delta}{m_{2j-1}}$, which for sufficiently large $j$ yields contradiction and thus proves that $\dist(\R y_i, Ty_i)\to 0$. 

Passing to a subsequence of $(y_i)$, find $a_0\in\R$ with
$\|Ty_i-a_0 y_i\|\to 0$. Thus for some infinite $J\subset\N$
the operator $(T-a_0 I_0)|_{\spa\{y_i:i\in
J\}}$ is compact. By Thm \ref{hi} the operator $T-a_0I_0$ is horizontally strictly singular. 
\end{proof}

We shall need next the following strong property of horizontally strictly singular operators, not obvious as $(X_s)_s$ do not form the Schauder decomposition of $X$. We profit here from the already shown bounded completeness of the basic sequence $(I_s)_{s=1}^\infty$. 

\begin{proposition}\label{hss-as-summable}
For any horizontally strictly singular $Q\in \mathscr{L}( X)$ let $(a_s)_{s=1}^\infty\subset\R$ be the scalars picked by Prop. \ref{op-xs} (3). Then the series $\sum_{s=1}^\infty a_sI_s$ is convergent in $\mathscr{L}( X)$. 
\end{proposition}

\begin{proof} 
We show first that $a_s\to 0$ as $s\to\infty$. Assume that $|a_s|>\delta$ for some $\delta>0$ and $s\in J$, $J\subset \N$ infinite. Let $a\neq 0$ be a cluster point of $(a_s)_{s\in J}$.  Choose $(s_n)_n\subset J$ with $|a_{s_n}-a|<a/8^n$ for all $n\in \N$.  Pick a normalized block sequence $(x_n)_n$ with $x_n\in  X_{s_n}$ (thus $(x_n)_n$ is h-block) and $\|Qx_n-a_{s_n}x_n\|<a/8^n$, thus $\|Qx_n-ax_n\|<a/4^n$ for each $n\in \N$. As $Q$ is horizontally strictly singular, there is a normalized h-block sequence $(z_i)_i$ of $(x_n)_n$ with $Qz_i\to 0$. On the other hand $\|Qz_i-az_i\|<a/2$, which yields the contradiction.

Now we prove that  $(\sum_{s=1}^Ma_sI_s)_M$ is bounded in $\mathscr{L}(X)$. This ends the proof by Cor. \ref{main1}. 

Fix $M\in\N$. Pick $M\leq N\in\N$ big enough to ensure $M|a_s|\leq 1$ for $s\geq N$. By Remark \ref{xs-remark} and Lemma \ref{ts-exact}  pick an $N$-g-special pair $((x_1,\dots,x_{2N} ), (f_1,\dots,f_{2N}))$ with $\supp(x_1)>M_N$ and $\|(Q-a_sI_s)x_s\|\leq 1/2^s$ for each $s=1,\dots,M, N+1,\dots, N+M$. 
By Prop. \ref{crucial} (3) for some $(b_s)_{s=1}^M\subset [-1,1]$ the vector $x=\sum_{s=1}^Mb_s(x_{2s-1}-x_{2s})$ satisfies $\|x\|\leq 1$ and  $\|\sum_{s=1}^Ma_sI_s\|\leq 40\|\sum_{s=1}^Ma_sI_sx\|$. On other hand \begin{align*}
    \|Q\|&\geq \|Qx\|= \|\sum_{s=1}^MQb_sx_s+\sum_{s=N+1}^{N+M}Qb_{s-N}x_s\|
    \\
    &\geq\|\sum_{s=1}^Ma_sb_sx_s\|-\|\sum_{s=1}^M(Q-a_sI_s)b_sx_s\|
    -\|\sum_{s=N+1}^{N+M}a_sb_{s-N}x_s\|-\|\sum_{s=N+1}^{N+M}(Q-a_sI_s)b_{s-N}x_s\|
    \\
    &\geq \|\sum_{s=1}^Ma_sI_sx\|-1-3M\max_{N+1\leq s\leq N+M}|a_s|-1
    \\
    &\geq \|\sum_{s=1}^Ma_sI_sx\|-5
\end{align*}
It follows that $\|\sum_{s=1}^Ma_sI_s\|\leq 40(\|Q\|+5)$, which ends the proof. 
\end{proof}

\begin{corollary}\label{main2}
  $\mathscr{L}(X)=\overline{\spa }\{I_s: s=0,1,2,\dots\}\oplus \mathscr{SS}(X)$.
\end{corollary}
\begin{proof}
For arbitrary $T\in\mathscr{L}( X)$ pick suitable $(a_s)_{s=0}^\infty$, $a_0$ by Thm \ref{scalar+hss}, $(a_s)_{s=1}^\infty$ by Prop. \ref{op-xs} for $T-a_0I_0$.  
Let $R=\sum_{s=0}^\infty a_sI_s\in\mathscr{L}( X)$, well-defined by Prop. \ref{hss-as-summable}. 

By the choice of $(a_s)_{s=1}^\infty$ for any $s\in\N$ the operator $(T-R)\circ I_s=(T-a_0I_0)\circ I_s-a_sI_s$ is strictly singular. Therefore $T-R\in\mathscr{SS}(X)$ by Prop. \ref{op-ss}. Therefore Thm \ref{main0} (1) ends the proof (notice that  $\dist(I_0, \overline{\spa }\{I_s: s=1,2,\dots\}+\mathscr{SS}(X))=1$). \end{proof}

The main theorem follows from Corollaries \ref{main1} and \ref{main2}.

\ 
\end{document}